\documentclass[11pt]{article}
\usepackage[margin=1in]{geometry}
\usepackage[utf8]{inputenc}
\usepackage[T1]{fontenc}
\usepackage[english]{babel}
\usepackage{lmodern}
\usepackage{graphicx}
\usepackage{wrapfig}
\usepackage[subrefformat=parens,labelformat=parens]{subfig}
\usepackage{caption}
\usepackage{paralist}
\usepackage{enumitem}  
\usepackage{float}
\usepackage[linesnumbered,ruled,vlined]{algorithm2e}

\graphicspath{{./figures/}}

\makeatletter
\renewcommand{\p@enumii}{}
\makeatother

\newcommand{\titel}{On minors of non-hamiltonian graphs}

\usepackage[usenames]{color}
\definecolor{hellblau}{rgb}{0.2,0.4,1} 
\definecolor{dunkelblau}{rgb}{0,0,0.8}
\definecolor{dunkelgruen}{rgb}{0,0.5,0}

\usepackage[
pdftex,
colorlinks,
linkcolor=dunkelblau,
urlcolor=dunkelblau,
citecolor=dunkelgruen,
bookmarks=true,
linktocpage=true,
pdfauthor={On-Hei Solomon Lo},
pdfsubject={}
]{hyperref}
\usepackage{pdfpages}
\usepackage{amsmath}
\usepackage{amsthm}
\allowdisplaybreaks
\usepackage{amsfonts}

\theoremstyle{plain}
\newtheorem{satz}{Satz}  
\newtheorem{theorem}[satz]{Theorem}
\newtheorem{lemma}[satz]{Lemma}

\newtheorem{proposition}[satz]{Proposition}

\newtheorem{claim}[satz]{Claim}

\theoremstyle{remark}

\theoremstyle{definition}

\newtheorem{conjecture}[satz]{Conjecture}
\newtheorem{question}[satz]{Question}

\numberwithin{satz}{section}

\begin{document}
	\title{\titel}
	\author{
		On-Hei Solomon Lo\thanks {School of Mathematical Sciences, Tongji University, Shanghai 200092, China}
		}
		
	\date{}
	
	\maketitle

\begin{abstract}
	A theorem of Tutte states that every 4-connected non-hamiltonian graph contains \( K_{3,3} \) as a minor. We strengthen this result by proving that such a graph must contain \( K_{3,4} \) as a minor, thereby confirming a special case of a conjecture posed by Chen, Yu, and Zang in a strong form.
	
	This result may be viewed as a step toward characterizing the minor-minimal 4-connected non-hamiltonian graphs. As a 3-connected analog, Ding and Marshall conjectured that every 3-connected non-hamiltonian graph has a minor of \( K_{3,4} \), \( \mathfrak{Q}^+ \), or the Herschel graph, where \( \mathfrak{Q}^+ \) is obtained from the cube by adding a new vertex adjacent to three independent vertices. We confirm this conjecture.
\end{abstract}

\sloppy

\section{Introduction}

In 1956, Tutte~\cite{Tutte1956} proved that every 4-connected planar graph is hamiltonian. This groundbreaking theorem has inspired a wealth of subsequent research. Notably, Thomas and Yu~\cite{Thomas1994} extended it to projective-planar graphs. However, a longstanding conjecture, proposed independently by Gr\"{u}nbaum~\cite{Gruenbaum1970} and Nash-Williams~\cite{Nash-Williams1973}, remains largely unresolved: Every 4-connected toroidal graph is hamiltonian. For further developments on extensions to graphs embedded in surfaces, we refer the reader to~\cite{Ozeki2021}. In this paper, we consider Tutte’s theorem from a different perspective.

A classical theorem due to Hall~\cite{Hall1943} and Wagner~\cite{Wagner1937} establishes that every 3-connected non-planar graph is either isomorphic to \( K_5 \) or contains a \( K_{3,3} \) minor. Consequently, Tutte’s theorem is equivalent to asserting that every 4-connected non-hamiltonian graph contains a \( K_{3,3} \) minor. Another well-known result of Wagner~\cite{Wagner1936} shows that every 4-connected non-planar graph contains a \( K_5 \) minor. Together with Tutte’s theorem, this implies that every 4-connected non-hamiltonian graph also contains a \( K_5 \) minor. These results motivate our investigation of minors in non-hamiltonian graphs.

Let $\mathcal{G}_k$ be the class of $k$-connected non-hamiltonian graphs. A \emph{minor-minimal $k$-connected non-hamiltonian graph} is a graph in $\mathcal{G}_k$ that does not contain any other graph from $\mathcal{G}_k$ as a minor. Characterizing minor-minimal 4-connected non-hamiltonian graphs enhances the aforementioned results, as it allows one to verify that every such graph contains both a $K_{3,3}$ minor and a $K_5$ minor. In pursuit of this, we answer the following question posed by Ding and Marshall~\cite{Ding2018} affirmatively, thereby improving the theorem of Tutte.

\begin{question}[Ding and Marshall~\cite{Ding2018}] \label{que:DM}  
	Does every $4$-connected non-hamiltonian graph contain a minor of $K_{3,4}$?  
\end{question}  

It is worth noting that the connectivity requirement in Tutte’s theorem cannot be relaxed: there exist 3-connected non-hamiltonian graphs that do not contain a \( K_{3,3} \) minor. Despite this, Chen and Yu~\cite{Chen2002} proved that every 3-connected graph \( G \) with no \( K_{3,3} \) minor contains a cycle of length at least \( \alpha \cdot |V(G)|^{\log_3 2} \) for some constant \( \alpha > 0 \). Moreover, the exponent \( \log_3 2 \) is known to be optimal. This result was later extended by Chen, Yu, and Zang~\cite{Chen2012}, who showed that for every integer \( t > 3 \), any 3-connected graph \( G \) with no \( K_{3,t} \) minor contains a cycle of length at least \( \alpha(t) \cdot |V(G)|^\beta \), where $\alpha(t) = (1/2)^{t(t-1)}$ and $\beta = \log_{1729} 2$. They further conjectured that a linear bound should hold in the 4-connected case. 

\begin{conjecture}[Chen, Yu, and Zang~\cite{Chen2012}]
	For each integer \( t \ge 4 \), there exists a function \( \alpha(t) > 0 \) such that every 4-connected graph $G$ with no \( K_{3,t} \) minor contains a cycle of length at least \( \alpha(t) \cdot |V(G)| \).
\end{conjecture}

Our solution to Question~\ref{que:DM} confirms Chen, Yu, and Zang's conjecture for the case \( t = 4 \), with \( \alpha(4) = 1 \).

The task of characterizing minor-minimal 4-connected non-hamiltonian graphs appears to be highly challenging, in stark contrast to the trivial case of the 2-connected analog. As noted in~\cite{Ding2018}, \( K_{2,3} \) is the unique minor-minimal 2-connected non-hamiltonian graph. 

Ding and Marshall~\cite{Ding2018} propose the following conjecture for characterizing the minor-minimal 3-connected non-hamiltonian graphs. The graphs \( K_{3,4} \), \( \mathfrak{Q}^+ \), and the Herschel graph \( \mathfrak{H} \) are depicted in Figure~\ref{fig:minimal}. It is straightforward to verify that these graphs are non-hamiltonian, as each is bipartite with an odd number of vertices.

\begin{figure}[!ht]
	\centering
	\begin{tabular}{ccc}
		\subfloat[The graph $K_{3, 4}$.\label{subfig:K34}]{
			\includegraphics[scale=1]{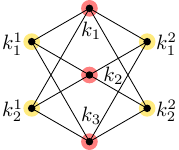}
		} & 
		\subfloat[The graph $\mathfrak{Q}^+$.\label{subfig:Q+}]{
			\includegraphics[scale=1]{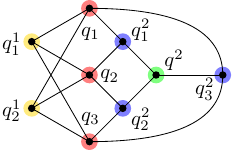}
		} & 
		\subfloat[The Herschel graph $\mathfrak{H}$.\label{subfig:Herschel}]{
			\includegraphics[scale=1]{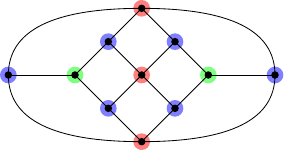}
		} 
	\end{tabular}
	\caption{Three 3-connected non-hamiltonian graphs.}
	\label{fig:minimal}
\end{figure}

\begin{conjecture}[Ding and Marshall~\cite{Ding2018}] \label{con:DM}
	Every $3$-connected non-hamiltonian graph contains a minor of $K_{3, 4}$, $\mathfrak{Q}^+$, or $\mathfrak{H}$.
\end{conjecture}

It was shown in~\cite{Lo2024} that Conjecture~\ref{con:DM} holds for planar graphs. In this paper, we settle the conjecture.  

It remains an open question how to characterize minor-minimal 4-connected non-hamiltonian graphs. In view of the 2- and 3-connected cases, it is tempting to conjecture that every minor-minimal 4-connected non-hamiltonian graph $G$ has an odd number of vertices and is bipartite, with one part consisting of $(|V(G)| + 1)/2$ vertices of degree 4.

To address Conjecture~\ref{con:DM}, we restrict our attention to the non-planar case, as the planar instance has already been resolved in~\cite{Lo2024}. We begin with the projective-planar setting. By applying a beautiful result of Robertson~\cite{Maharry2016}, we find that any minor-minimal counterexample \( G \) must contain a Wagner graph minor. It is well known that, up to symmetry, such a minor admits only two distinct embeddings in the projective plane, which offers a helpful starting point for our analysis.

For the non-projective-planar case, we rely on a refinement of Archdeacon’s theorem~\cite{Archdeacon1981}, established by Robertson, Seymour, and Thomas~\cite{Ding2014}. In our setting, this result implies that \( G \) must contain a minor of one of three specific graphs from the Archdeacon list. Building on this, we apply a powerful tool developed by Johnson and Thomas~\cite{Johnson2002} concerning minors of internally 4-connected graphs. We then proceed through a sequence of structural arguments to locate a particular spanning subdivision in \( G \), from which we are able to construct a Hamilton cycle, leading to a contradiction. A similar strategy is employed in the proof of Question~\ref{que:DM}, where, by virtue of the theorem of Thomas and Yu~\cite{Thomas1994}, we may assume from the outset that any counterexample is non-projective-planar.

Our structural results on graphs excluding \( K_{3,4} \) minors may be of independent interest. In particular, we show that every 4-connected non-projective-planar graph with more than eight vertices and no \( K_{3,4} \) minor contains a spanning subdivision of one of two specific graphs; see Sections~\ref{sec:DEEF} and~\ref{sec:DMQ} for details. Note that all planar graphs have no \( K_{3,4} \) minor, and the class of projective-planar graphs without a \( K_{3,4} \) minor has been well characterized~\cite{Maharry2012, Maharry2017a}. A natural direction for future research is to seek a full characterization of all graphs that exclude \( K_{3,4} \) as a minor.

The paper is organized as follows. In Section~\ref{sec:pre}, we review basic definitions and several key tools. Section~\ref{sec:projective-planar} establishes the projective-planar case of Conjecture~\ref{con:DM}. We then focus on the non-projective-planar case, where the core structural analysis of graphs excluding certain minors is presented in Section~\ref{sec:DEEF}. Finally, we complete the proofs of Conjecture~\ref{con:DM} and Question~\ref{que:DM} in Sections~\ref{sec:DMC} and~\ref{sec:DMQ}, respectively.

\section{Preliminaries}\label{sec:pre}

A graph \( H \) is a \emph{minor} of a graph \( G \) if it can be obtained from \( G \) by deleting edges, deleting vertices, and contracting edges. 
It is well known that \( G \) contains \( H \) as a minor if and only if there exists a mapping \( \mu \) that maps \( V(H) \) into pairwise disjoint subsets of \( V(G) \) and maps \( E(H) \) to pairwise internally disjoint paths in \( G \), such that for each \( v \in V(H) \), the induced subgraph \( G[\mu(v)] \) is connected, and for any \( uv \in E(H) \), the path \( \mu(uv) \) joins \( \mu(u) \) and \( \mu(v) \) with no internal vertex of \( \mu(uv) \) in \( \bigcup_{w \in V(H)} \mu(w) \). 

Clearly, if \( v \in V(H) \) has degree at most four, we can assume that \( G[\mu(v)] \) contains a spanning path \( P \), such that for any \( uv \in E(H) \), the intersection of \( \mu(v) \) and \( \mu(uv) \) is an end-vertex of \( P \), and each end-vertex of \( P \) is incident to at least two paths corresponding to edges incident to \( v \).

If \( \mu(v) \) consists of exactly one vertex for every \( v \in V(H) \), we say that \( G \) contains a \emph{subdivision} of \( H \). Note that if every vertex of \( H \) has degree three, then the existence of a minor of \( H \) is equivalent to the existence of a subdivision of \( H \).

A \emph{$k$-cut} in a connected graph $G$ is a subset $S$ of vertices such that $|S| = k$ and $G - S$ is disconnected. We say $G$ is $k$-connected if $|V(G)| > k$ and $G$ has no $k'$-cut for $k' < k$. 
Moreover, a 3-connected graph \( G \) on at least five vertices is \emph{internally $4$-connected} if either $G$ is isomorphic to $K_{3,3}$, or every 3-cut of \( G \) is an independent set that separates \( G \) into a single vertex and a remaining component.

For any path \( P \) and any vertices \( u, v \) in \( P \), we denote the subpath with end-vertices \( u \) and \( v \) by \( P[u, v] \). A vertex of $P$ is \emph{internal} if it is not an end-vertex. We denote by \( P(u, v) \) the (possibly empty) path obtained from \( P[u, v] \) by deleting $u$ and $v$.

Let \( G \) be a graph embedded on a compact surface, and let \( C \) be a contractible cycle in \( G \), which is also a simple closed curve on the surface. Deleting \( C \) from the surface results in two regions. We call the disk region the \emph{interior} and the other region the \emph{exterior} of \( C \). However, when we say that a subgraph $H$ of $G$ lies in the interior (or exterior) of $C$, some vertices of $H$ may still belong to $C$. We can impose an orientation on the vertices of \( C \). For any vertices \( u, v \) on \( C \), denote by \( C[u, v] \) the path in \( C \) from \( u \) to \( v \); and by \( C(u, v) \) the (possibly empty) path obtained from \( C[u, v] \) by deleting the vertices \( u \) and \( v \).

Let $H$ be a subgraph of $G$. Let \( A \) be a component of \( G - V(H) \). The \emph{bridge} \( B \) of $H$ containing \( A \) consists of \( A \), the vertices in \( H \) that are adjacent to \( A \), as well as the edges connecting \( A \) to \( H \). The vertices in \( V(H) \cap V(B) \) are called the \emph{attachments} of \( B \). (Note that in this paper, we do not consider an edge with both end-vertices in \( H \) but not in \( H \) itself as a bridge.)

For any positive integer $k$, denote by $[k]$ the set of positive integers at most $k$.

\subsection{Robertson's characterization of graphs without Wagner minors}

Recall that the \emph{Wagner graph} is defined as the graph obtained from a cycle of length eight by adding four edges, each joining a pair of vertices at distance four in that cycle. The following is the classical characterization of graphs without a minor of the Wagner graph by Robertson.

\begin{theorem}[\cite{Maharry2016}] \label{thm:Wagner}
	Let $G$ be an internally $4$-connected graph without any minor of the Wagner graph. Then one of the following holds:
	\begin{itemize}
		\item $G$ is planar.
		\item $G$ has at most seven vertices.
		\item $G$ is isomorphic to the line graph of $K_{3,3}$.
		\item $G$ contains four vertices such that every edge of $G$ is incident to at least one of these vertices.
		\item $G$ contains two vertices such that $G$ becomes a cycle when these two vertices are removed.
	\end{itemize}
\end{theorem}

\subsection{Seymour's splitter theorem}

Let \( H \) be a graph and \( v \in V(H) \). Let \( H' \) be a graph obtained from \( H \) by deleting \( v \), adding two adjacent vertices \( v_1 \) and \( v_2 \), and joining each neighbor of \( v \) to exactly one of \( v_1 \) and \( v_2 \), such that each of \( v_1 \) and \( v_2 \) is joined to at least two neighbors of \( v \). Thus, \( v \) necessarily has degree at least 4. We say that \( H' \) is obtained from \( H \) by \emph{splitting} \( v \).

The famous Seymour’s splitter theorem is stated as follows. Recall that a \emph{wheel} is a graph obtained from a cycle by adding a new vertex and joining it to all vertices of the cycle.

\begin{theorem}[\cite{Seymour1980}]\label{thm:Seymour}
	Let \( G \) and \( H \) be $3$-connected graphs. Suppose that \( G \) contains \( H \) as a minor and that if \( H \) is a wheel, then it is the largest wheel minor of \( G \). Then a graph isomorphic to \( G \) can be obtained from \( H \) by repeatedly applying the operations of adding an edge between two non-adjacent vertices and splitting a vertex.
\end{theorem}

\subsection{Subdivisions and stable bridges in internally 4-connected graphs}

One can easily prove that if \( G \) contains \( H \) as a minor but does not contain any graph obtained from \( H \) by splitting a vertex as a minor, then \( G \) contains a subdivision of \( H \).

Let \( G \) and \( H \) be graphs such that \( G \) contains a subdivision of \( H \). There exists a mapping \( \eta \) that maps the vertices of \( H \) to distinct vertices of \( G \) and maps \( E(H) \) to pairwise internally disjoint paths in \( G \), such that, for any \( uv \in E(H) \), the path \( \eta(uv) \) has end-vertices \( \eta(u) \) and \( \eta(v) \), with no internal vertex of \( \eta(uv) \) contained in \( \{\eta(v) : v \in V(H)\} \). 

Let \( \eta(H) \) denote the subgraph of \( G \) formed by the image of \( \eta \). We also say that \( \eta(H) \) is a subdivision of \( H \) in \( G \). Define \( \eta(V(H)) := \{ \eta(v) : v \in V(H) \} \). The paths \( \eta(e) \) for \( e \in E(H) \) are called \emph{segments} of \( \eta \). For \( v \in V(H) \), the \emph{bark} of \( v \) with respect to \( \eta \) is the set consisting of \( \eta(v) \) and the internal vertices of all segments that have \( \eta(v) \) as an end-vertex.

An \emph{unstable fragment} is one of the following:
\begin{itemize}
	\item the union of \( \eta(v) \) and \( \eta(e) \) for some \( v \in V(H) \) and \( e \in E(H) \),
	\item the union of two segments with a common end-vertex,
	\item the union of three segments sharing a common end-vertex of degree 3 in \( \eta(H) \).
\end{itemize}
A bridge of \( \eta(H) \) in \( G \) is \emph{unstable} if its attachments are contained in an unstable fragment; otherwise, it is \emph{stable}.

Johnson and Thomas~\cite{Johnson2002} introduced a special type of subdivision called ``lexicographically maximal,'' which we refer to as a \emph{JT-subdivision}. For the technical definition of a JT-subdivision, we refer to \cite{Johnson2002}. The key fact is that if \( G \) contains a subdivision of \( H \), then \( G \) contains a JT-subdivision of \( H \).

The following lemma from \cite[(5.2)]{Johnson2002} will be useful in our proof. Note that, due to differing definitions of bridges, the following statement appears slightly different from that in \cite{Johnson2002}.

\begin{lemma}[\cite{Johnson2002}] \label{lem:JTstable}
	Let \( G \) and \( H \) be internally $4$-connected graphs such that \( G \) contains a JT-subdivision \( \eta \) of \( H \). If \( G \) does not contain any graph obtained from \( H \) by splitting a vertex as a minor, then every bridge of \( \eta(H) \) is stable and every segment of $\eta$ is an induced path in $G$.
\end{lemma}

We also need the following lemma from \cite[(7.2)]{Johnson2002}

\begin{lemma}[\cite{Johnson2002}] \label{lem:JTedge}
	Let \( G \) and \( H \) be internally $4$-connected graphs such that \( G \) contains a JT-subdivision \( \eta(H) \) of \( H \). Let $u$ be a vertex of H of degree three, and let $e_1$ and $e_2$ be two distinct edges of $H$ incident with $u$. If there exists an edge of $G$ joining a vertex $v$ in $\eta(e_1) - \eta(u)$ and an internal vertex of $\eta(e_2)$, then the path $\eta(e_1)[\eta(u), v]$ has only one edge.
\end{lemma}

\subsection{Properties of minor-minimal counterexamples}

A graph is a \emph{minor-minimal} counterexample to Conjecture~\ref{con:DM} if it is a counterexample, and no proper minor of it is a counterexample. Ding and Marshall~\cite{Ding2018} proved that such a graph is internally 4-connected.

\begin{lemma}[\cite{Ding2018}] \label{lem:i4c}
	Every minor-minimal counterexample to Conjecture~\ref{con:DM} is internally $4$-connected.
\end{lemma}

The following two lemmas from~\cite{Lo2024}, together with the preceding lemma, will be used frequently in the proof of Conjecture~\ref{con:DM}.

\begin{lemma}[\cite{Lo2024}] \label{lem:na4}
	Every minor-minimal counterexample to Conjecture~\ref{con:DM} contains no pair of adjacent vertices each with degree at least $4$.
\end{lemma}

\begin{lemma}[\cite{Lo2024}] \label{lem:n3c}
	Every minor-minimal counterexample to Conjecture~\ref{con:DM} does not contain any cycle of length $3$.
\end{lemma}

\section{Wagner minors and Hamilton cycles in projective-planar graphs} \label{sec:projective-planar}

In this section, we prove the projective-planar case of Ding and Marshall's conjecture. Throughout, we assume that $G$ is a minor-minimal counterexample to Conjecture~\ref{con:DM} and that $G$ can be embedded on the projective plane. Our goal is to derive a contradiction by showing that $G$ contains a Hamilton cycle.

We first show that $G$ must contain a Wagner minor.

\begin{claim} \label{cla:Wagner}
	$G$ contains a minor of the Wagner graph.
\end{claim}
\begin{proof}
	By Lemma~\ref{lem:i4c}, $G$ is internally 4-connected. Since it is proven in \cite{Lo2024} that $G$ is non-planar, it suffices to consider the four remaining situations stated in Theorem~\ref{thm:Wagner}.
	
	Suppose $G$ has at most seven vertices. Let $C$ be a longest cycle in $G$. Since $G$ is not hamiltonian, there exists a vertex \( v_0 \in V(G) \setminus V(C) \). Clearly, there are three paths joining \( v_0 \) to \( C \), such that they only intersect at \( v_0 \). Moreover, the other ends of these three paths are pairwise non-adjacent in \( C \) because \( C \) is a longest cycle in \( G \). Thus, \( C \) must have length exactly 6. We can write \( C := v_1 v_2 v_3 v_4 v_5 v_6 v_1 \), where \( v_0 \) is adjacent to \( v_1 \), \( v_3 \), and \( v_5 \). By replacing \( C \) with the cycle \( v_1 v_2 v_3 v_0 v_5 v_6 v_1 \), we can conclude that \( v_1 v_4 \in E(G) \). Similarly, we have \( v_3 v_6, v_5 v_2 \in E(G) \). Therefore, \( G \) contains \( K_{3, 4} \) as a subgraph, which is a contradiction. Hence, we may assume \( |V(G)| \geq 8 \).
	
	Note that the line graph of \( K_{3, 3} \) is hamiltonian. Therefore, \( G \) is not isomorphic to the line graph of \( K_{3, 3} \).
	
	Suppose \( G \) contains four vertices \( v_1, v_2, v_3, v_4 \) such that every edge is incident to at least one of these four vertices. For any \( \{w_1, w_2, w_3, w_4\} \subseteq V(G) \setminus \{v_1, v_2, v_3, v_4\} \), it is not hard to show that the graph induced by the edges joining \( \{v_1, v_2, v_3, v_4\} \) and \( \{w_1, w_2, w_3, w_4\} \) contains a subgraph isomorphic to the cube. Therefore, \( G \) is hamiltonian if \( |V(G)| = 8 \), and \( G \) contains \( \mathfrak{Q}^+ \) as a subgraph if \( |V(G)| > 8 \); either case leads to a contradiction.
	
	Finally, it is known that if \( G \) contains two vertices \( w_1 \) and \( w_2 \) whose removal results in a cycle \( C \), then either all vertices of \( C \) are adjacent to both \( w_1 \) and \( w_2 \), or \( C \) has an even length, denoted by \( C := v_1 v_2 \dots v_{2t} v_1 \), such that \( v_1, v_3, \dots, v_{2t-1} \) are joined to \( w_1 \) and \( v_2, v_4, \dots, v_{2t} \) are joined to \( w_2 \) (cf.\ \cite{Maharry2016}). In either case, \( G \) is hamiltonian, leading to a contradiction.
	
	Therefore, by Theorem~\ref{thm:Wagner}, \( G \) must contain a minor of the Wagner graph.
\end{proof}

We now define the Wagner graph \( V_8 \) as the graph with vertex set \( \{v_1, v_2, \dots, v_8\} \), where each \( v_i \) (for \( i \in [8] \)) has precisely three neighbors: \( v_{i-1} \), \( v_{i+1} \), and \( v_{i+4} \), with indices taken modulo \( 8 \). 

By Claim~\ref{cla:Wagner}, there exists a mapping \( \eta \) that maps the vertices of \( V_8 \) to distinct vertices of \( G \) and maps \( E(V_8) \) into pairwise internally disjoint paths in \( G \). For any \( uv \in E(V_8) \), the path \( \eta(uv) \) joins \( \eta(u) \) and \( \eta(v) \), with no internal vertex of \( \eta(uv) \) lying in \( \bigcup_{w \in V(V_8)} \eta(w) \). 

We call \( \eta(v_i v_{i+4}) \) a \emph{chord path}. 

Denote by \( O^\eta \) the cycle in \( G \) formed by the union of \( \eta(v_1 v_2), \eta(v_2 v_3), \dots, \eta(v_7 v_8), \eta(v_8 v_1) \).

The vertices of \( V_8 \) and those of \( O^\eta \) are naturally ordered cyclically.

We fix an embedding of \( G \) on the projective plane. It is well known that \( O^\eta \) must be a contractible cycle, with at most one chord path lying in its interior (cf.~\cite{Maharry2017}). Whenever possible, we choose a subdivision \( \eta \) such that all chord paths lie in the exterior of \( O^\eta \). Given this, we further choose \( \eta \) so that the interior of \( O^\eta \) is inclusionwise maximal.

If both \( \eta(v_i v_{i+4}) \) and \( \eta(v_{i+1} v_{i+5}) \) lie in the exterior of \( O^\eta \), let \( O^\eta_i \) denote the cycle in \( G \) formed by the union of \( \eta(v_i v_{i+4}) \), \( \eta(v_i v_{i+1}) \), \( \eta(v_{i+4} v_{i+5}) \), and \( \eta(v_{i+1} v_{i+5}) \). If \( \eta(v_i v_{i+4}) \) lies in the exterior of \( O^\eta \) while \( \eta(v_{i+1} v_{i+5}) \) lies in the interior, let \( O^\eta_i \) denote the cycle in \( G \) formed by the union of \( \eta(v_i v_{i+4}) \), \( \eta(v_i v_{i+1}) \), \( \eta(v_{i+1} v_{i+2}) \), \( \eta(v_{i+4} v_{i+5}) \), \( \eta(v_{i+5} v_{i+6}) \), and \( \eta(v_{i+2} v_{i+6}) \). In either case, \( O^\eta_i \) is a contractible cycle.

\begin{claim} \label{cla:emptyext}
	The exterior of \( O^\eta \) contains no vertex of \( G \). Hence, every bridge of \( O^\eta \) lies within the interior of \( O^\eta \).
\end{claim}

\begin{proof}
	We first show that every chord path \( \eta(v_i v_{i+4}) \) lying in the exterior of \( O^\eta \) has no internal vertex. Suppose, to the contrary, that \( \eta(v_i v_{i+4}) \) contains some internal vertex. Then, since \( G \) is 3-connected, there exists a path \( P \) internally disjoint from \( \eta(V_{8}) \) that connects an internal vertex \( v \) of \( \eta(v_i v_{i+4}) \) to a vertex \( w \) not in \( \eta(v_i v_{i+4}) \). Without loss of generality, we may assume that \( w \) lies in \( O^\eta_i \).
	
	Suppose \( \eta(v_{i+1} v_{i+5}) \) also lies in the exterior of \( O^\eta \). We can construct another subdivision \( \eta' \) of \( V_{8} \) such that the interior of \( O^{\eta'} \) properly contains that of \( O^\eta \), as follows. If \( w \) is an internal vertex of \( \eta(v_{i+1} v_{i+5}) \), we define \( \eta'(v_i) \) as \( v \), \( \eta'(v_{i+1}) \) as \( w \), \( \eta'(v_i v_{i+1}) \) as \( P \), \( \eta'(v_{i-1} v_i) \) as the union of \( \eta(v_{i-1} v_i) \) and \( \eta(v_{i} v_{i+4})[\eta(v_i), v] \), \( \eta'(v_{i+1} v_{i+2}) \) as the union of \( \eta(v_{i+1} v_{i+2}) \) and \( \eta(v_{i+1} v_{i+5})[\eta(v_{i+1}), w] \), \( \eta'(v_{i} v_{i+4}) \) as \( \eta(v_{i} v_{i+4})[v, \eta(v_{i+4})] \), and \( \eta'(v_{i+1} v_{i+5}) \) as \( \eta(v_{i+1} v_{i+5})[w, \eta(v_{i+5})] \). Moreover, \( \eta' \) agrees with \( \eta \) on \( (V(V_{8}) \cup E(V_{8})) \setminus \{ v_i, v_{i+1}, v_i v_{i+1}, v_{i-1} v_i, v_{i+1} v_{i+2}, v_i v_{i+4}, v_{i+1} v_{i+5} \} \). This construction contradicts the choice of \( \eta \) that the interior of $O^\eta$ is inclusionwise maximal. 
	Similar constructions apply to the cases where \( w \in V(\eta(v_{i} v_{i+1})) \setminus \{ \eta(v_i) \} \) or \( w \in V(\eta(v_{i+4} v_{i+5})) \setminus \{ \eta(v_{i+4}) \} \); the details are easy and hence omitted.
	
	Suppose \( \eta(v_{i+1} v_{i+5}) \) lies in the interior of \( O^\eta \). If \( w \in (V(\eta(v_i v_{i+1})) \cup V(\eta(v_{i+4} v_{i+5}))) \setminus \{v_i, v_{i+4}\} \), then, similarly to the previous paragraph, we can construct another subdivision of \( V_{8} \) to contradict the maximality of the interior of \( O^\eta \). Otherwise, \( G \) would contain a minor of the second graph given in Figure~\ref{fig:emptyext} and hence a \( K_{3,4} \) minor, which is impossible. This proves that all chord paths in the exterior of $O^\eta$ have no internal vertices.
	
	If there exists a vertex \( v \) in the exterior of \( O^\eta \), then \( v \) must lie in the interior of some \( O^\eta_i \) (in this case, \( \eta(v_i v_{i+1}) \) lies in the exterior of \( O^\eta \)). Since \( G \) is 3-connected, \( v \) is joined to three distinct vertices \( u_1, u_2, u_3 \) in \( O^\eta_i \) via three internally disjoint paths.
	
	Suppose \( \eta(v_{i+1} v_{i+5}) \) also lies in the exterior of \( O^\eta \). Then, we can assume \( u_1 \) and \( u_2 \) are in \( \eta(v_i v_{i+1}) \), which results in a path internally disjoint from \( \eta(V_{8}) \) that joins two vertices of \( \eta(v_i v_{i+1}) \) (passing through \( v \)). This leads to a contradiction, as it implies the existence of another subdivision \( \eta'(V_8) \) of \( V_8 \), such that the interior of \( O^{\eta'} \) properly contains the interior of \( O^\eta \).
	
	Now suppose \( \eta(v_{i+1} v_{i+5}) \) lies in the interior of \( O^\eta \). In this case, we can assume \( u_1 \) and \( u_2 \) are in the union of \( \eta(v_i v_{i+1}) \) and \( \eta(v_{i+1} v_{i+2}) \). If \( \{u_1, u_2\} \neq \{\eta(v_i), \eta(v_{i+2})\} \), then, again, we can obtain a subdivision of \( V_{8} \) that contradicts the maximality of the interior of \( O^\eta \). Consequently, we must have \( \{u_1, u_2\} = \{\eta(v_i), \eta(v_{i+2})\} \), and \( u_3 \) lies in the union of \( \eta(v_{i+4} v_{i+5}) \) and \( \eta(v_{i+5} v_{i+6}) \). 
	
	If \( u_3 \neq \eta(v_{i+5}) \), then \( G \) contains the second graph shown in Figure~\ref{fig:emptyext} as a minor. Otherwise, \( G \) contains the third graph shown in Figure~\ref{fig:emptyext} as a minor. In either case, \( G \) has a \( K_{3,4} \) minor, which is a contradiction.
\end{proof}

\begin{figure}[!ht]
	\centering{%
		\includegraphics[scale=.8]{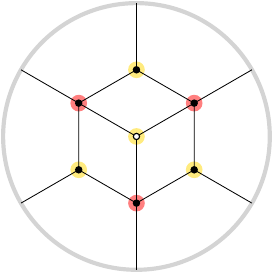}
	}
	\hfill
	{%
		\includegraphics[scale=.8]{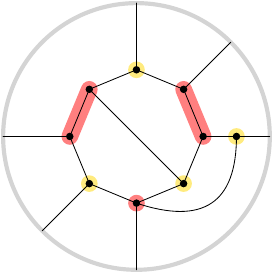}
	}
	\hfill
	{%
		\includegraphics[scale=.8]{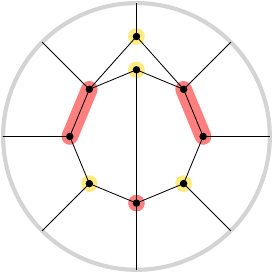}
	}
	\hfill
	{%
		\includegraphics[scale=.8]{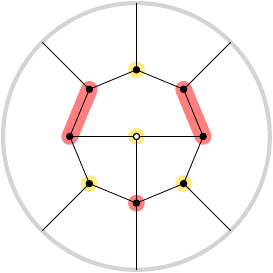}
	}
	\caption{The graph \( K_{3,4} \) and three graphs, each containing a \( K_{3,4} \) minor, embedded on the projective plane, with the gray circle representing the crosscap.}
	\label{fig:emptyext}
\end{figure}

In the following, we examine the structure of bridges of \( O^\eta \). There are two types of bridges, depending on whether they contain a chord path. 

Let \( B \) be a bridge of \( O^\eta \). Since \( G \) is 3-connected, \( B \) has at least three attachments. Let \( u_1, u_2, u_3 \) be three attachments of \( B \) appearing in \( O^\eta \) in this order (so that \( O^\eta[u_1, u_2] \) does not contain \( u_3 \)). We call \( O^\eta(u_1, u_2) \), \( O^\eta(u_2, u_3) \), and \( O^\eta(u_3, u_1) \) the \emph{sectors} associated with these attachments. A sector is \emph{sung} if it contains a vertex from \( \eta(V(V_8)) \). 

\begin{claim} \label{cla:widebridge}
	Let \( B \) be a bridge of \( O^\eta \) that does not contain any chord path. Suppose \( u_1, u_2, u_3 \) are three attachments of \( B \) appearing in \( O^\eta \) in this order. Then, at least one of the sectors associated with \( u_1, u_2, u_3 \) is not sung.
\end{claim}

\begin{proof}
	By Claim~\ref{cla:emptyext}, \( B \) lies within the interior of \( O^\eta \).
	
	Suppose that every sector is sung. Let \( H \) be obtained from the union of \( \eta(V_{8}) \) and \( B \) by contracting \( V(B) \setminus V(O^\eta) \) to a single vertex \( v \).
	We will derive a contradiction by showing that \( H \) contains a minor of either \( K_{3,4} \) or \( \mathfrak{Q}^+ \).
	
	We first consider the case when \( \{u_1, u_2, u_3\} \subset \eta(V(V_{8})) \). 
	
	We show that there exists \( i \in [8] \) such that \( \eta(v_i) \in \{u_1, u_2, u_3\} \), \( \eta(v_{i+4}) \notin \{u_1, u_2, u_3\} \), and each of \( O^\eta(\eta(v_i), \eta(v_{i+4})) \) and \( O^\eta(\eta(v_{i+4}), \eta(v_i)) \) contains one vertex from \( \{u_1, u_2, u_3\} \). If there exists \( j \in [8] \) such that \( \{u_1, u_2, u_3\} \subset V(O^\eta[\eta(v_j), \eta(v_{j+4})]) \), then, assuming \( u_1, u_2, u_3 \) appear in \( O^\eta[\eta(v_j), \eta(v_{j+4})] \) in this order, we take \( i \in [8] \) with \( \eta(v_i) = u_2 \). Otherwise, we choose any \( i \in [8] \) such that \( \eta(v_i) \in \{u_1, u_2, u_3\} \). In either case, this yields the desired configuration.

	Without loss of generality, assume \( u_2 = \eta(v_i) \). We have \( u_1 \) lies in \( O^\eta(\eta(v_{i+4}), \eta(v_{i-1})) \) and \( u_3 \) lies in \( O^\eta(\eta(v_{i+1}), \eta(v_{i+4})) \).
	
	Now, label the vertices of \( K_{3,4} \) as shown in Figure~\ref{subfig:K34}. Define \( \mu \) as follows: \( \mu(k_1) := \{u_2\} \), \( \mu(k_2) := V(O^\eta(\eta(v_{i+4}), \eta(v_{i-1}))) \), \( \mu(k_3) := V(O^\eta(\eta(v_{i+1}), \eta(v_{i+4}))) \), \( \mu(k^1_1) := \{v\} \), \( \mu(k^1_2) := \{v_{i+4}\} \), \( \mu(k^2_1) := \{v_{i-1}\} \), and \( \mu(k^2_2) := \{v_{i+1}\} \). \( \mu \) does realize a minor of \( K_{3,4} \), although we omit defining the connecting paths, as they can be easily determined. This is illustrated in the first graph in Figure~\ref{fig:emptyext}, where the lowest vertex represents \( u_2 \), the central vertex represents \( v \), and so on.
	
	Next, consider the case where not all of \( u_1, u_2, u_3 \) are in \( \eta(V(V_{8})) \). Recall that our goal is to show that $H$ contains a minor of \( K_{3,4} \) or \( \mathfrak{Q}^+ \). If any vertex \( w \in \{u_1, u_2, u_3\} \) lies in \( O^\eta(v_i, v_{i+1}) \) for some \( i \in [8] \), we can contract \( O^\eta[v_i, w] \) or \( O^\eta[w, v_{i+1}] \), treating the contracted path as an attachment, provided the property that every sector is sung is preserved. More precisely, if \( O^\eta[v_{i-1}, v_i] \) (respectively, \( O^\eta[v_{i+1}, v_{i+2}] \)) contains none of \( u_1, u_2, u_3 \), then we can contract \( O^\eta[v_i, w] \) (respectively, \( O^\eta[w, v_{i+1}] \)) to obtain a new graph and proceed to find the desired minor in it.
	
	If no path contraction is possible yet not all attachments are in $\eta(V(V_8))$, then it is not hard to see that \( H \) contains the first graph shown in Figure~\ref{fig:wide} as a minor. (The drawing may be different if there is a chord path lying in the interior of $O^\eta$.) One can readily see that this graph contains \( \mathfrak{Q}^+ \) as a minor, which leads to a contradiction.
\end{proof}

\begin{figure}[!ht]
	\centering{%
		\includegraphics[scale=.8]{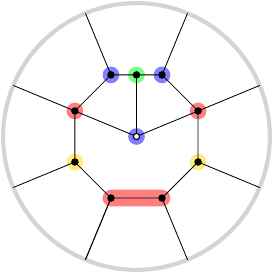}
	}
	\hfill
	{%
		\includegraphics[scale=.8]{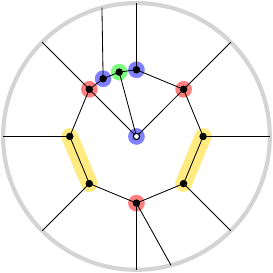}
	}
	\hfill
	{%
		\includegraphics[scale=.8]{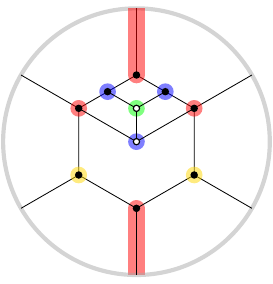}
	}
	\hfill
	{%
		\includegraphics[scale=.8]{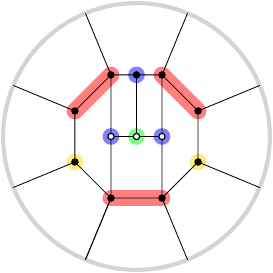}
	}
	\caption{Four graphs, each containing a minor of \( \mathfrak{Q}^+ \), embedded on the projective plane.}
	\label{fig:wide}
\end{figure}

\begin{claim} \label{cla:widebridge1}
	Let \( B \) be a bridge of \( O^\eta \) that does not contain any chord path. Then, there exists a subdivision $\tilde{\eta}(V_8)$ of \( V_8 \) with \( O^\eta = O^{\tilde{\eta}} \) and three attachments of \( B \) such that precisely two sectors are sung.
\end{claim}
\begin{proof}
	
	Let \( u_1, u_2, u_3 \) be three attachments of \( B \) appearing in \( O^\eta \) in this order.  
	
	By Claim~\ref{cla:widebridge}, some sector associated with \( u_1, u_2, u_3 \) is not sung. Since at least one sector must be sung, the claim follows immediately if there are two sung sectors. 
	
	Without loss of generality, assume \( u_1 \) and \( u_2 \) lie in \( \eta(v_1 v_2) \).  
	
	Suppose \( u_3 \) also lies in \( \eta(v_1 v_2) \). 
	
	If \( \eta(v_1 v_2) \) does not contain all attachments of \( B \), then any attachment outside \( \eta(v_1 v_2) \), together with \( u_1 \) and \( u_2 \) or \( u_2 \) and \( u_3 \), yields two sung sectors. Thus, we assume that all attachments of \( B \) are contained within \( \eta(v_1 v_2) \). 
	
	Let \( w_1 \) and \( w_2 \) be the attachments of \( B \) such that \( O^\eta[w_1, w_2] \) is a subpath of \( \eta(v_1 v_2) \) containing all attachments of \( B \).  
	
	If both \( \eta(v_1 v_5) \) and \( \eta(v_2 v_6) \) lie in the exterior of \( O^\eta \), then by the 3-connectedness of \( G \), the maximality of the interior of \( O^\eta \), and Claim~\ref{cla:emptyext}, there exists an edge \( e \) with one end-vertex \( x_1 \) in \( O^\eta(w_1, w_2) \) and the other \( x_2 \) in \( \eta(v_5 v_6) \). Moreover, if not all vertices of \( O^\eta[w_1, w_2] \) are attachments, we choose \( e \) such that \( x_1 \) is not an attachment.  
	
	There exists a subdivision \( \eta' \) of \( V_8 \) such that \( \eta'(V_8) \) is obtained from \( \eta(V_8) \) by deleting one of \( \eta(v_1) \eta(v_5) \) or \( \eta(v_2) \eta(v_6) \) and adding \( e \).  
	
	If \( x_1 \) is not an attachment, we select \( w_1 \), \( w_2 \), and any other attachment \( w_3 \). Clearly, under \( \eta' \), there are at least two sung sectors associated with \( w_1, w_2, w_3 \).  
	
	If \( x_1 \) is an attachment, then by Lemma~\ref{lem:i4c}, \( B \) has at least four attachments. Selecting \( w_1 \), \( w_2 \), and another attachment distinct from \( w_1, w_2, x_1 \) again yields two sung sectors.  
	
	If one of \( \eta(v_1 v_5) \) or \( \eta(v_2 v_6) \) lies in the interior of \( O^\eta \), say \( \eta(v_1 v_5) \), then by the arguments above, there exists an edge \( e \) with one end-vertex \( x_1 \) in \( O^\eta(w_1, w_2) \) and the other \( x_2 \) in \( O^\eta[\eta(v_4), \eta(v_6)] \). Note that \( e \) can always be chosen such that \( x_2 \neq \eta(v_4) \); otherwise, \( \{w_1, w_2, \eta(v_4)\} \) would be a 3-cut violating Lemma~\ref{lem:i4c}. Furthermore, if \( B \) has exactly three attachments, we can always further require that \( x_1 \) is a non-attachment.

	If \( x_2 \) lies in \( O^\eta(\eta(v_4), \eta(v_6)) \), then there exists a subdivision of \( V_8 \) where all chord paths lie in the exterior, contradicting our choice of \( \eta \). 
	
	If \( x_2 = \eta(v_6) \), we can proceed as in the case where both \( \eta(v_1 v_5) \) and \( \eta(v_2 v_6) \) lie in the exterior of \( O^\eta \).  
	
	Thus, we may assume \( u_3 \) is not in \( \eta(v_1 v_2) \).  
	
	Suppose the sector \( O^\eta(u_2, u_3) \) is not sung. Then \( u_2 = \eta(v_2) \), \( u_3 \) lies in \( \eta(v_2 v_3) \), and \( \eta(v_2 v_6) \) lies in the exterior of \( O^\eta \). If there is an attachment of \( B \) in \( O^\eta(u_1, u_2) \) or \( O^\eta(u_2, u_3) \), replacing \( u_2 \) with this attachment yields two sung sectors. Thus, we assume no such attachment exists. It is also easy to show that \( u_1, u_2, u_3 \) are the only attachments of \( B \).  
	
	Without loss of generality, \( \eta(v_1 v_5) \) lies in the exterior of \( O^\eta \). By Lemma~\ref{lem:i4c}, \( u_1 \) is not adjacent to \( u_2 \). Thus, there exists an edge \( e \) joining \( O^\eta(u_1, u_2) \) with \( \eta(v_5 v_6) \), yielding another subdivision with two sung sectors associated with \( u_1, u_2, u_3 \). 
	
	Therefore, we can conclude that \( O^\eta(u_2, u_3) \) is sung. Symmetrically, \( O^\eta(u_3, u_1) \) is also sung. This proves the claim.
\end{proof}

\begin{claim} \label{cla:widebridge2}
	Let \( B \) be a bridge of \( O^\eta \) that does not contain any chord path. Then there exist a subdivision \( \tilde{\eta}(V_8) \) of \( V_8 \) with \( O^\eta = O^{\tilde{\eta}} \), four attachments \( \tilde{u}_1, \tilde{u}_2, \tilde{u}_3, \tilde{u}_4 \) of \( B \), such that \( \tilde{u}_1, \tilde{u}_2 \) are in \( \tilde{\eta}(v_1 v_2) \), \( \tilde{u}_3, \tilde{u}_4 \) are not in \( \tilde{\eta}(v_1 v_2) \), and there are two sung sectors associated with \( \tilde{u}_1, \tilde{u}_2, \tilde{u}_3 \).
\end{claim}

\begin{proof}
	By Claim~\ref{cla:widebridge1}, there exist attachments \( u_1, u_2, u_3 \) of \( B \), occurring in \( O^\eta \) in this order, such that two sectors are sung. Thus, the claim almost holds, except that the existence of \( \tilde{u}_4 \) has not yet been ensured.
	
	Without loss of generality, assume \( O^\eta[u_1, u_2] \) is a subpath of \( \eta(v_1 v_2) \), and \( \eta(v_2 v_6) \) lies in the exterior of \( O^\eta \).
	
	Suppose the claim does not hold for \( \eta \). Then there are attachments \( w_1 \) and \( w_2 \) of \( B \) such that \( O^\eta[w_1, w_2] \) is a subpath of \( \eta(v_1 v_2) \) and contains all attachments of \( B \) except \( u_3 \).  
	
	We claim that there exists an edge \( e \) joining \( O^\eta(w_1, w_2) \) and \( \eta(v_5 v_6) \) if \( \eta(v_1 v_5) \) lies in the exterior of \( O^\eta \), and joining \( O^\eta(w_1, w_2) \) to \( \eta(v_6) \) if \( \eta(v_1 v_5) \) lies in the interior of \( O^\eta \). 
	
	If \( \eta(v_1 v_5) \) lies in the exterior of \( O^\eta \) and such an edge \( e \) does not exist, then \( \{w_1, w_2, u_3\} \) would form a 3-cut of \( G \). It follows from Lemma~\ref{lem:i4c} that \( B \) has precisely three attachments \( w_1, w_2, u_3 \), and \( w_1 \) and \( w_2 \) are non-adjacent. Thus, the desired edge \( e \) must exist. 
	
	For the same reason, if \( \eta(v_1 v_5) \) lies in the interior of \( O^\eta \), there is an edge \( e \) joining \( O^\eta(w_1, w_2) \) and \( O^\eta[\eta(v_4), \eta(v_6)] \). By the choice of \( \eta \), we have \( e \) connects \( O^\eta(w_1, w_2) \) to \( \eta(v_4) \) or \( \eta(v_6) \). If \( e \) is incident to \( \eta(v_4) \), then \( G \) contains one of the first three graphs given in Figure~\ref{fig:wide3} and hence \( K_{3,4} \) as a minor, which is impossible. Thus, \( e \) is incident to \( \eta(v_6) \).
	
	Let \( x_1 \in V(O^\eta(w_1, w_2)) \) and \( x_2 \in V( \eta(v_5 v_6)) \) be the end-vertices of \( e \). For the case where \( \eta(v_1 v_5) \) lies in the exterior of \( O^\eta \), it causes no loss of generality to assume \( x_2 \neq \eta(v_5) \). Define \( \eta' \) as the subdivision of \( V_8 \) such that \( \eta'(V_8) \) is obtained from \( \eta(V_8) \) by deleting \( \eta(v_2)\eta(v_6) \) and adding \( e \).  
	
	Consider \( \eta' \) and the attachments \( w_1, w_2, u_3 \). It is clear that \( O^{\eta'}(w_1, w_2) \) is sung, and precisely one of \( O^{\eta'}(u_3, w_1) \) or \( O^{\eta'}(w_2, u_3) \) is sung, by Claim~\ref{cla:widebridge}.  
	
	Assume \( O^{\eta'}(w_2, u_3) \) is not sung. Then \( u_3 \) must be in \( \eta(v_2 v_3) - \eta(v_2) \).  
	
	If there exists an attachment \( w \) of \( B \) in \( O^\eta(x_1, \eta(v_2)) \), then we can show that \( G \) contains a minor of \( \mathfrak{Q}^+ \) as follows. Adopt the vertex labeling of \( \mathfrak{Q}^+ \) as given in Figure~\ref{subfig:Q+}. Define \( \mu(q_1) := V(O^\eta[\eta(v_1), w_1]) \), \( \mu(q_2) := V(O^\eta[u_3, \eta(v_3)]) \), \( \mu(q_3) := V(O^\eta[x_2, \eta(v_6)]) \), \( \mu(q^1_1) := V(\eta(v_4 v_5)) \), \( \mu(q^1_2) := V(\eta(v_7 v_8)) \), \( \mu(q^2_1) := V(B) \setminus V(O^\eta) \), \( \mu(q^2_2) := \{\eta(v_2)\} \), \( \mu(q^2_3) := \{x_1\} \), and \( \mu(q^2) := \{w\} \). As illustrated in the second graph in Figure~\ref{fig:wide}, \( \mu \) realizes a minor of \( \mathfrak{Q}^+ \). This leads to a contradiction.  
	
	Thus, \( O^\eta(x_1, \eta(v_2)) \) contains no attachment of \( B \), and \( w_2 = \eta(v_2) \). Moreover, as \( O^\eta(u_2, u_3) \) is sung, we have $u_2 \neq w_2$, and \( u_1 \) and \( u_2 \) are in \( O^\eta[w_1, x_1] \). 
	
	If \( u_2 \neq x_1 \), then \( \eta' \) (with \( \eta'(v_1) = \eta(v_1) \)) yields the desired configuration by taking \( u_1' := u_1 \), \( u_2' := u_2 \), \( u_3' := u_3 \), and \( u_4' := w_2 \). 
	
	If \( u_2 = x_1 \), then, by Lemma~\ref{lem:na4}, \( u_2 \) and \( w_2 \) are not adjacent. In this case, there exists an edge \( e' \) joining \( O^\eta(u_2, w_2) \) and \( O^\eta[x_2, \eta(v_6)] \). Let \( \eta'' \) be the subdivision of \( V_8 \) such that \( \eta''(V_8) \) is obtained from \( \eta(V_8) \) by deleting \( \eta(v_2)\eta(v_6) \) and adding \( e' \). Similarly, one can readily obtain the desired configuration consisting of \( \eta'' \), \( u_1 \), \( u_2 \), \( u_3 \), and \( w_2 \).  
	
Assume that \( O^{\eta'}(u_3, w_1) \) is not sung. Then \( u_3 \) lies in \( \eta(v_8 v_1) - \eta(v_1) \) and \( w_1 = \eta(v_1) \). Moreover, since \( O^{\eta}(u_3, u_1) \) is sung, \( u_1 \) is in \( O^\eta(w_1, w_2) \). 

	We consider the subdivision \( \eta' \) with \( \eta'(v_1) = \eta(v_8) \) (and hence \( x_1 = \eta'(v_3) \)). Set \( u_1' := u_3 \), \( u_2' := w_1 \), \( u_3' := w_2 \), and \( u_4' := u_1 \). It is not hard to verify that this yields the desired configuration.  
\end{proof}

\begin{figure}[!ht]
	\centering{%
		\includegraphics[scale=.8]{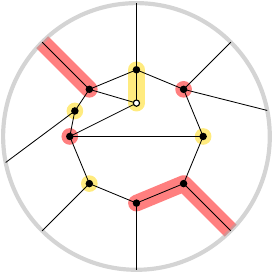}
	}
	\hfill
	{%
		\includegraphics[scale=.8]{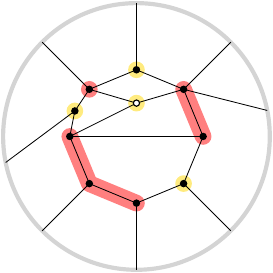}
	}
	\hfill
	{%
		\includegraphics[scale=.8]{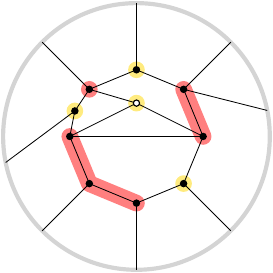}
	}
	\hfill
	{%
		\includegraphics[scale=.8]{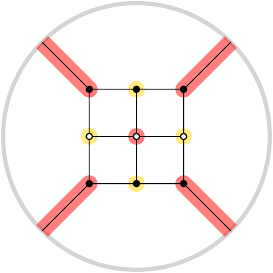}
	}
	\caption{Four graphs, each containing a minor of \( K_{3,4} \), embedded on the projective plane.}
	\label{fig:wide3}
\end{figure}

\begin{claim} \label{chordbridge}
	Suppose \( B \) is a bridge of \( O^\eta \) that contains a chord path \( \eta(v_i v_{i+4}) \). Then, \( B \) has no attachment in \( O^\eta(v_{i+1}, v_{i+3}) \) or in \( O^\eta(v_{i+5}, v_{i-1}) \).
\end{claim}
\begin{proof}
	Claim~\ref{cla:emptyext} asserts that \( B \) lies within the interior of \( O^\eta \).  
	If \( B \) has an attachment in \( O^\eta(v_{i+1}, v_{i+3}) \) or \( O^\eta(v_{i+5}, v_{i-1}) \), then \( G \) would contain the fourth graph in Figure~\ref{fig:emptyext} as a minor, which is impossible.
\end{proof} 

The following claim is crucial for finding paths used in constructing a Hamilton cycle.

\begin{claim} \label{cla:widebridge3}
	Let \( B \) be a bridge of \( O^\eta \). Then \( B \) has precisely four attachment, \( w_1, w_2, w_3, w_4 \), such that \( w_1 \) is adjacent to \( w_2 \) and \( w_3 \) is adjacent to \( w_4 \) in \( O^\eta \). Moreover, the union of \( B \) and the edges \( w_1w_2 \) and \( w_3w_4 \) contains two disjoint paths: one with end-vertices \( w_1 \) and \( w_2 \), and another with end-vertices \( w_3 \) and \( w_4 \). Together, these two paths span the union of \( B \) and the edges \( w_1w_2 \) and \( w_3w_4 \).  
\end{claim}
\begin{proof}
	We first consider that $B$ does not contain any chord path. We apply Claim~\ref{cla:widebridge2} to the bridge \( B \). Without loss of generality, \( B \) has attachments \( u_1, u_2, u_3, u_4 \) such that \( u_1, u_2 \) are in \( \eta(v_1 v_2) \), \( u_3, u_4 \) are not in \( \eta(v_1 v_2) \), the sectors \( O^\eta(u_2, u_3) \) and \( O^\eta(u_3, u_1) \) are sung, and \( u_4 \) is in \( O^\eta(\eta(v_2), u_3) \).
	
	The sector \( O^\eta(u_4, u_3) \) is not sung; otherwise, the three sectors \( O^\eta(u_1, u_4) \), \( O^\eta(u_4, u_3) \), and \( O^\eta(u_3, u_1) \) would all be sung, violating Claim~\ref{cla:widebridge}.
	
	Suppose \( u_3 \) is in \( \eta(v_2 v_3) - \eta(v_2) \). So, the chord path $\eta(v_2 v_6)$ lies in the exterior of $O^\eta$, and \( u_4 \) is in \( \eta(v_2 v_3) - \eta(v_2) - \eta(v_3) \). Moreover, $u_1$ and $u_2$ are in \( \eta(v_1 v_2) - \eta(v_2) \) as \( O^\eta(u_2, u_3) \) is sung. Claim~\ref{cla:widebridge} implies that all attachments of \( B \) must be contained in \( O^\eta[\eta(v_1), \eta(v_3)] \). 
	
	Let \( w_1, w_2, w_3, w_4 \) be the attachments of \( B \) such that \( O^\eta[w_1, w_2] \) is a subpath of \( \eta(v_1 v_2) - \eta(v_2) \) containing all attachments of \( B \) in \( \eta(v_1 v_2) - \eta(v_2) \), and \( O^\eta[w_4, w_3] \) is a subpath of \( \eta(v_2 v_3) - \eta(v_2) \) containing all attachments of \( B \) in \( \eta(v_2 v_3) - \eta(v_2) \).
	
	We show that \( B - \eta(v_2) \) contains two disjoint paths, one joining \( w_1 \) to \( w_3 \) and the other joining \( w_2 \) to \( w_4 \). Suppose this is not the case. Then there exists a vertex \( x \) that separates \( \{w_1, w_2\} \) from \( \{w_3, w_4\} \) in \( B - \eta(v_2) \). Clearly, \( x \notin \{w_1, w_2, w_3, w_4\} \).  
	
	We observe that \( O^\eta(w_1, w_2) \) is not incident to any edge in the exterior of \( O^\eta \). If \( \eta(v_1 v_5) \) lies in the exterior and there exists an edge \( e \) joining \( O^\eta(w_1, w_2) \) to \( \eta(v_5 v_6) \), then either \( \eta(v_6) \) is not incident to \( e \), in which case there exists a subdivision where \( e \) replaces the edge \( \eta(v_1) \eta(v_5) \), resulting in the three sectors associated with \( w_1, w_2, \) and \( w_3 \) being sung, contradicting Claim~\ref{cla:widebridge}; or \( \eta(v_6) \) is incident to \( e \), in which case \( G \) contains a minor of the second graph given in Figure~\ref{fig:wide}. A similar argument applies when \( \eta(v_1 v_5) \) lies in the interior and an edge \( e \) joins \( O^\eta(w_1, w_2) \) to \( O^\eta[\eta(v_4), \eta(v_6)] \). The only situation we have not yet considered is when \( \eta(v_4) \) is incident to \( e \). However, in this case, \( G \) contains a minor of the first graph given in Figure~\ref{fig:wide3}, which is not possible.
	
Likewise, \( O^\eta(w_4, w_3) \) has no incident edges in the exterior of \( O^\eta \).

If both \( w_1 \) and \( w_2 \) are adjacent to \( x \), then by Lemma~\ref{lem:n3c}, they must be non-adjacent. Moreover, since \( O^\eta(w_1, w_2) \) does not connect to the exterior of \( O^\eta \), the set \( \{w_1, w_2, x\} \) would form a 3-cut, contradicting Lemma~\ref{lem:i4c}. Therefore, at least one of \( w_1 \) and \( w_2 \) is not adjacent to \( x \) but to some other vertex \( y \) in \( V(B) \setminus V(O^\eta) \). If \( \{w_1, w_2, x\} \) is a 3-cut, then it would separate \( V(O^\eta(w_1, w_2)) \cup \{y\} \) from \( \{w_3, w_4\} \), contradicting Lemma~\ref{lem:i4c}. This implies that \( \eta(v_2) \) is an attachment of \( B \).

	Applying Lemma~\ref{lem:i4c}, we conclude that \( \{w_1, \eta(v_2), x\} \) is not a 3-cut. 
	
	Consequently, \( O^\eta(w_1, \eta(v_2)) \) must be incident to some edge \( e_1 \) in the exterior of \( O^\eta \), and similarly, \( O^\eta(\eta(v_2), w_3) \) must be incident to some edge \( e_2 \) in the exterior of \( O^\eta \). Since \( G \) does not contain the first graph in Figure~\ref{fig:wide3} as a minor, we conclude that \( \eta(v_4) \) is not incident to \( e_1 \) and \( \eta(v_8) \) is not incident to \( e_2 \). Moreover, we can assume that at least one of \( e_1 \) and \( e_2 \) is not be incident to \( \eta(v_6) \), as \( \{w_1, w_3, \eta(v_6)\} \) is not a 3-cut. This allows us to construct a subdivision, using $e_1$ and $e_2$, in which the three sectors associated with \( w_1, \eta(v_2), \) and \( w_3 \) are sung, contradicting Claim~\ref{cla:widebridge}.  
	
Thus, we deduce that \( B - \eta(v_2) \) contains the desired paths, implying that \( G \) has the third graph in Figure~\ref{fig:wide} as a minor. However, this is impossible because that graph contains \( \mathfrak{Q}^+ \) as a minor. We conclude that \( u_3 \) cannot belong to \( \eta(v_2 v_3) - \eta(v_2) \).

Suppose \( u_3 \) belongs to \( \eta(v_3 v_4) - \eta(v_3) \). We have that \( u_4 \) is in \( \eta(v_3 v_4) - \eta(v_4) \). By Claim~\ref{cla:widebridge}, every attachment of \( B \) is contained in either \( \eta(v_1 v_2) \) or \( \eta(v_3 v_4) \).

Let \( w_1, w_2, w_3, w_4 \) be the attachments of \( B \) such that \( O^\eta[w_1, w_2] \) is a subpath of \( \eta(v_1 v_2) \) containing all attachments of \( B \) in \( \eta(v_1 v_2) \), and \( O^\eta[w_4, w_3] \) is a subpath of \( \eta(v_3 v_4) \) containing all attachments of \( B \) in \( \eta(v_3 v_4) \). By similar arguments as before, one can show that \( B \) contains two disjoint paths: one joining \( w_1 \) to \( w_3 \) and the other joining \( w_2 \) to \( w_4 \); we omit the details.  

If \( w_2 = \eta(v_2) \) and \( w_4 = \eta(v_3) \), then, by Lemma~\ref{lem:na4}, $\eta(v_2)$ and $\eta(v_3)$ are not adjacent, and hence there exists an edge joining \( \eta(v_2 v_3) - \eta(v_2) - \eta(v_3) \) and \( \eta(v_6 v_7) \). We can use that edge to modify the subdivision $\eta$, and hence it causes no loss of generality to assume \( w_2 \neq \eta(v_2) \) or \( w_4 \neq \eta(v_3) \). As a result, \( G \) contains the third graph given in Figure~\ref{fig:wide} as a minor, a contradiction. We conclude that \( u_3 \) does not belong to \( \eta(v_3 v_4) - \eta(v_3) \).

By symmetry, it remains to consider the case where \( u_3 \) and \( u_4 \) are in the path \( \eta(v_i v_{i+1}) \), where \( i \in \{4, 5\} \). By Claim~\ref{cla:widebridge}, every attachment of \( B \) is contained in either \( \eta(v_1 v_2) \) or \( \eta(v_i v_{i+1}) \).  

Let \( w_1, w_2, w_3, w_4 \) be the attachments of \( B \) such that \( O^\eta[w_1, w_2] \) is a subpath of \( \eta(v_1 v_2) \) containing all attachments of \( B \) in \( \eta(v_1 v_2) \), and \( O^\eta[w_4, w_3] \) is a subpath of \( \eta(v_i v_{i+1}) \) containing all attachments of \( B \) in \( \eta(v_i v_{i+1}) \).

As before, one can readily show that there is no edge joining \( O^\eta(w_1, w_2) \) or \( O^\eta(w_4, w_3) \) to the exterior of \( O^\eta \) (regardless of whether there is a chord path lying in the interior of \( O^\eta \)). This implies that the vertices in \( O^\eta[w_1, w_2] \) and \( O^\eta[w_4, w_3] \) are precisely the attachments of \( B \), and the union of \( B \), \( O^\eta[w_1, w_2] \), and \( O^\eta[w_4, w_3] \), denoted by \( \bar{B} \), is 2-connected. Let \( P_1 \) and \( P_2 \) be the disjoint paths in \( B \) such that \( P_1 \) has end-vertices \( w_1 \) and \( w_3 \), \( P_2 \) has end-vertices \( w_2 \) and \( w_4 \), and the union of \( P_1 \), \( P_2 \), \( O^\eta[w_1, w_2] \), and \( O^\eta[w_4, w_3] \) bounds \( \bar{B} \).  

Suppose \( B' := \bar{B} - V(P_1) - V(P_2) \) is non-empty. Then, as \( G \) is 3-connected, every component of \( B' \) must join both \( P_1 \) and \( P_2 \). There exists a path \( H_0 \) in $\bar{B}$, with end-vertices $y_1$ and $y_2$, satisfying \( y_1 \in V(P_1) \), \( y_2 \in V(P_2) \), \( |V(H_0)| \geq 3 \), and \( H_0 \) is internally disjoint from \( P_1 \) and \( P_2 \). We take such a path \( H_0 \) so that the (possibly empty) interior bounded by the union of \( O^\eta[w_1, w_2] \), \( P_1[w_1, y_1] \), \( P_2[w_2, y_2] \), and \( H_0 \) is inclusionwise minimal. Note that \( H_0 \) is \( O^\eta[w_1, w_2] \) when \( O^\eta[w_1, w_2] \) has some internal vertex. Moreover, there is no edge incident to an internal vertex of \( H_0 \) lying in the interior bounded by the union of \( O^\eta[w_1, w_2] \), \( P_1[w_1, y_1] \), \( P_2[w_2, y_2] \), and \( H_0 \). From this, one can easily show that \( H_0 \) is disjoint from \( O^\eta[w_4, w_3] \).  

Let \( H_1 \) be the subgraph obtained from the union of the facial cycles corresponding to the faces that are incident to \( H_0 \) and lie in the interior bounded by the union of \( O^\eta[w_4, w_3] \), \( P_1[w_3, y_1] \), \( P_2[w_4, y_2] \), and \( H_0 \), by removing \( V(H_0) \). By Lemma~\ref{lem:i4c}, we have that $H_1$ is a path. Moreover, \( H_1 \) has one end-vertex \( y_3 \) in \( P_1 \) and one end-vertex \( y_4 \) in \( P_2 \), and is internally disjoint from \( P_1 \) and \( P_2 \). Furthermore, $H_0 - y_1 - y_2$ and $H_1 - y_3 - y_4$ are in the same component of $B'$.

If \( H_1 \) is disjoint from \( O^\eta[w_4, w_3] \), we deduce a contradiction, namely that \( G \) contains a minor of \( \mathfrak{Q}^+ \), as follows. Adopt the vertex labeling of \( \mathfrak{Q}^+ \) as given in Figure~\ref{subfig:Q+}. Define \( \mu(q_1) := V(P_1[w_1, y_1]) \cup V(O^\eta[\eta(v_{i+3}), w_1]) \), \( \mu(q_2) := V(P_2[w_2, y_2]) \cup V(O^\eta[w_2, \eta(v_{i-2})]) \), \( \mu(q_3) := V(\eta(v_i v_{i+1})) \), \( \mu(q^1_1) := \{\eta(v_{i-1})\} \), \( \mu(q^1_2) := \{\eta(v_{i+2})\} \), \( \mu(q^2_1) := V(H_0) \setminus \{y_1, y_2\} \), \( \mu(q^2_2) := \{y_4\} \), \( \mu(q^2_3) := \{y_3\} \), and \( \mu(q^2) := V(H_1) \setminus \{y_3, y_4\} \). Clearly, \( \mu \) realizes a minor of \( \mathfrak{Q}^+ \). The configuration for \( i = 5 \) is illustrated in the fourth graph in Figure~\ref{fig:wide}.

Thus, we conclude that \( H_1 \) intersects \( O^\eta[w_4, w_3] \). It follows immediately that \( H_0 \) is not \( O^\eta[w_1, w_2] \), and hence \( w_1 \) and \( w_2 \) are adjacent in \( O^\eta \). By symmetry, we also have that \( w_3 \) and \( w_4 \) are adjacent in \( O^\eta \). (In the following, we will demonstrate how to find the two desired paths: one joining \( w_1 \) and \( w_2 \), and the other joining \( w_3 \) and \( w_4 \). These arguments will also be applicable when \( B \) contains a chord path, with slight modifications.)

Similarly, one can show that \( H_0 \) intersects \( O^\eta[w_1, w_2] \); otherwise, a minor of \( \mathfrak{Q}^+ \) would exist. By Lemma~\ref{lem:na4}, \( H_0 \) intersects \( O^\eta[w_1, w_2] \) at exactly one vertex. Without loss of generality, we assume that this intersection occurs at \( w_1 \). Thus, \( w_1 = y_1 \), \( w_2 \) is adjacent to \( y_2 \) in \( P_2 \), and, by Lemma~\ref{lem:n3c}, the union of the edges \( w_1 w_2 \) and \( w_2 y_2 \) together with the path \( H_0 \) bounds a face. As \( y_3 \) is an internal vertex of \( P_1 \), we have \( H_1 \) and \( O^\eta[w_4, w_3] \) intersect at \( w_4 \). This implies that \( y_2 \) is the only internal vertex of \( P_2 \), and symmetrically, \( y_3 \) is the only internal vertex of \( P_1 \). 

Note that \( w_4 \) has degree at least four. By Lemma~\ref{lem:na4}, \( y_3 \) is the only neighbor of \( w_3 \) that does not belong to \( O^\eta \).

Let \( C \) be the cycle consisting of the edges \( y_3 w_3 \) and \( w_3 w_4 \) along with the path \( H_1 \). Suppose there exists a vertex lying in the interior of \( C \). Let \( B_C \) be a bridge of \( C \) contained in the interior of $C$. The attachments of \( B_C \) must all lie on \( H_1 \). Let \( c_1 \) and \( c_2 \) be the attachments of \( B_C \) such that \(y_3, c_1, c_2, y_4 \) occur in \( H_1 \) in this order, and \( H_1[c_1, c_2] \) contains all the attachments of \( B_C \).

Suppose there exist distinct vertices \( c_3, c_4, c_5 \) other than $c_1, c_2$, such that \( c_1, c_3, c_4, c_5, c_2 \) occur in \( H_1 \) in this order, where \( c_4 \) is an attachment of \( B_C \), and each of \( c_3, c_5 \) is joined to \( H_0 \) by an edge. Define \( \mu(q_1) := V(H_0) \cup V(\eta(v_1 v_2)) \), \( \mu(q_2) := V(H_1[c_2, w_4]) \cup V(O^\eta[\eta(v_4), w_4]) \), \( \mu(q_3) := V(H_1[y_3, c_1]) \cup V(O^\eta[w_3, \eta(v_7)]) \), \( \mu(q^1_1) := \{\eta(v_3)\} \), \( \mu(q^1_2) := \{\eta(v_8)\} \), \( \mu(q^2_1) := \{c_5\} \), \( \mu(q^2_2) := V(B_C) \setminus V(H_1) \), \( \mu(q^2_3) := \{c_3\} \), and \( \mu(q^2) := \{c_4\} \). One can verify straightforwardly that \( \mu \) realizes a minor of \( \mathfrak{Q}^+ \) (in $G - \eta(v_2)\eta(v_6)$), which is impossible.

Thus, by Lemmas~\ref{lem:i4c} and~\ref{lem:n3c}, \( B_C \) has at least four attachments.

Suppose all vertices in \( H_1[c_1, c_2] \) are attachments. Then, by Lemma~\ref{lem:i4c}, there exist two vertices, say \( c_3 \) and \( c_5 \), in \( H_1(c_1, c_2) \), each of which is joined to \( H_0 \) by an edge. Furthermore, by Lemma~\ref{lem:na4}, there is a vertex \( c_4 \) in \( H_1(c_3, c_5) \). However, this configuration was previously shown to be impossible.  

Therefore, there exist vertices \( c_1', c_2' \) in \( H_1 \) such that \( c_1, c_1', c_2', c_2 \) appear in this order along \( H_1 \), where \( c_1' \) and \( c_2' \) are non-adjacent, and all vertices in \( H_1(c_1', c_2') \) are not attachments of $B_C$. In particular, at least one vertex in \( H_1(c_1', c_2') \) is joined to \( H_0 \) by an edge.  

If \( c_1 = c_1' \), then by Lemma~\ref{lem:i4c}, there must be a vertex in \( H_1(c_2', c_2) \) that is joined to \( H_0 \) by an edge, contradicting our previous observation. Thus, we must have \( c_1' \neq c_1 \) and \( c_2' \neq c_2 \). It is also straightforward to show that there exist disjoint paths \( U \) and \( U' \) in \( B_C \) such that \( U \) has end-vertices \( c_1 \) and \( c_2 \), while \( U' \) has end-vertices \( c_1' \) and \( c_2' \).  

Define \( \mu(q_1) := V(H_0) \cup V(\eta(v_1 v_2)) \cup V(H_1(c_1', c_2')) \), \( \mu(q_2) := V(H_1[c_2, w_4]) \cup V(O^\eta[\eta(v_4), w_4]) \), \( \mu(q_3) := V(H_1[y_3, c_1]) \cup V(O^\eta[w_3, \eta(v_7)]) \), \( \mu(q^1_1) := \{\eta(v_3)\} \), \( \mu(q^1_2) := \{\eta(v_8)\} \), \( \mu(q^2_1) := \{c_2'\} \), \( \mu(q^2_2) := V(U(c_1, c_2)) \), \( \mu(q^2_3) := \{c_1'\} \), and \( \mu(q^2) := V(U'(c_1', c_2')) \).  One can readily verify that \( \mu \) realizes a minor of \( \mathfrak{Q}^+ \) (in $G - \eta(v_2)\eta(v_6)$), which is a contradiction.  

Thus, we conclude that the interior of \( C \) contains no vertices. Consider the path consisting of \( H_0 \) together with the edge \( w_2 y_2 \), and the path consisting of \( H_1 \) together with the edge \( w_3 y_3 \). These two disjoint paths lie within \( \bar{B} \) and span \( \bar{B} \), making them the desired paths. This completes the case where \( B' \) is non-empty.  

If \( B' \) is empty, then letting \( w_1' \) and \( w_2' \) denote the respective neighbors of \( w_1 \) in \( P_1 \) and \( w_2 \) in \( P_2 \), it follows easily that \( w_1' \) and \( w_2' \) must be adjacent. Thus, the path consisting of the single edge \( w_1 w_2 \) and the path formed by \( P_1 - w_1 \), \( P_2 - w_2 \), and the edge \( w_1' w_2' \) constitute the desired paths.  

We have finished the proof for the case where $B$ does not contain any chord path.

Next, we consider the case where \( B \) contains a chord path, say \( \eta(v_2 v_6) \). By Claim~\ref{chordbridge}, every attachment is contained in either \( O^\eta[\eta(v_1), \eta(v_3)] \) or \( O^\eta[\eta(v_5), \eta(v_7)] \). Let \( w_1, w_2, w_4, w_3 \) be attachments of \( B \) occurring in \( O^\eta \) in this order, such that \( O^\eta[w_1, w_2] \) is a subpath of \( O^\eta[\eta(v_1), \eta(v_3)] \), \( O^\eta[w_4, w_3] \) is a subpath of \( O^\eta[\eta(v_5), \eta(v_7)] \), and every attachment of \( B \) is contained in either \( O^\eta[w_1, w_2] \) or \( O^\eta[w_4, w_3] \). It is easy to show that there is no edge joining \( O^\eta(w_1, w_2) \) or \( O^\eta(w_4, w_3) \) that lies in the exterior of \( O^\eta \). From this, it follows that \( w_1 \neq w_2 \) and \( w_4 \neq w_3 \).  

As before, denote by \( \bar{B} \) the union of \( O^\eta[w_1, w_2] \), \( O^\eta[w_4, w_3] \), and \( B \). Clearly, $\bar{B}$ is 2-connected. Let \( P_1 \) and \( P_2 \) be the disjoint paths in \( B \) such that \( P_1 \) has end-vertices \( w_1 \) and \( w_3 \), \( P_2 \) has end-vertices \( w_2 \) and \( w_4 \), and the union of \( P_1 \), \( P_2 \), \( O^\eta[w_1, w_2] \), and \( O^\eta[w_4, w_3] \) bounds \( \bar{B} \).  

Suppose \( w_1 \) and \( w_2 \) are not adjacent in \( O^\eta \).  

Let \( H_0 = O^\eta[w_1, w_2] \), and let \( H_1 \) be the path obtained from the union of the facial cycles corresponding to the faces that are incident to \( H_0 \) and lie in the interior bounded by the union of \( O^\eta[w_4, w_3] \), \( P_1[w_3, y_1] \), \( P_2[w_4, y_2] \), and \( H_0 \), by removing \( V(H_0) \). Then, \( H_1 \) has an end-vertex \( y_3 \) in \( P_1 \) and an end-vertex \( y_4 \) in \( P_2 \), and \( H_1 \) is internally disjoint from \( O^\eta[w_1, w_2] \), \( O^\eta[w_4, w_3] \), \( P_1 \), and \( P_2 \). Note that there is an edge joining \( H_0(w_1, w_2) \) and \( H_1(w_3, w_4) \).

If one of \( \eta(v_5) \) or \( \eta(v_7) \), say \( \eta(v_5) \), is not in \( O^\eta[w_4, w_3] \), then we define \( \mu(q_1) := V(O^\eta[\eta(v_1), w_1]) \), \( \mu(q_2) := V(O^\eta[w_2, \eta(v_4)]) \), \( \mu(q_3) := V(O^\eta[w_4, \eta(v_7)]) \), \( \mu(q^1_1) := \{\eta(v_5)\} \), \( \mu(q^1_2) := \{\eta(v_8)\} \), \( \mu(q^2_1) := V(H_0) \setminus \{w_1, w_2\} \), \( \mu(q^2_2) := \{y_4\} \), \( \mu(q^2_3) := \{y_3\} \), and \( \mu(q^2) := V(H_1) \setminus \{y_3, y_4\} \). This leads to a contradiction, as \( \mu \) realizes a minor of \( \mathfrak{Q}^+ \).  

If \( O^\eta[w_4, w_3] \) contains both \( \eta(v_5) \) and \( \eta(v_7) \), it follows that \( w_4 = \eta(v_5) \) and \( w_3 = \eta(v_7) \). Then there exists a path lying in the interior of the union of \( O^\eta[w_4, w_3] \), \( H_1 \), \( P_1[w_3, y_3] \), and \( P_2[w_4, y_4] \), with one end-vertex in \( O^\eta(w_4, w_3) \) and another end-vertex \( x \) in either \( H_1 \), \( P_1(w_3, y_3) \), or \( P_2(w_4, y_4) \). If \( x \) is not in \( H_1(y_3, y_4) \), we can obtain a minor of \( \mathfrak{Q}^+ \) similarly to the previous case. If \( x \) is in \( H_1(y_3, y_4) \), then \( G \) contains the fourth graph given in Figure~\ref{fig:wide3} and hence $K_{3,4}$ as a minor, which is a contradiction.  

We conclude that \( w_1 \) is adjacent to \( w_2 \) and \( w_4 \) is adjacent to \( w_3 \) in \( O^\eta \). Without loss of generality, \( w_1, w_2 \) lie in \( \eta(v_1 v_2) \) and \( w_4, w_3 \) in \( \eta(v_i v_{i+1}) \) for some \( i \in \{5, 6\} \). Now, we can proceed to find the two desired paths using the arguments we employed in the case where \( B \) does not contain a chord path; the details are left to the reader.

This completes the proof.  
\end{proof}

We are now ready to deduce the contradiction that \( G \) has a Hamilton cycle.

Let \( \mathcal{B} \) denote the set of bridges of \( O^\eta \). As established by Claim~\ref{cla:widebridge3}, each \( B \in \mathcal{B} \) has exactly four attachments \( w_1, w_2, w_3, w_4 \) such that \( w_1w_2, w_3w_4 \in E(O^\eta) \), and there exist disjoint paths \( U_1, U_2 \) within the union of \( B \) and the edges \( w_1w_2 \), \( w_3w_4 \), where \( U_1 \) connects \( w_1 \) and \( w_2 \), \( U_2 \) connects \( w_3 \) and \( w_4 \), and together they span the union of \( B \) and the edges $w_1 w_2, w_3 w_4$. We call \( w_1w_2, w_3w_4 \) the \emph{hinges} and \( U_1, U_2 \) the \emph{wings} of $B$.

By Claim~\ref{cla:emptyext}, no two distinct bridges share a common hinge. Thus, by removing the hinges and attaching the wings for all bridges in \( \mathcal{B} \) to the cycle \( O^\eta \), we obtain a Hamilton cycle of $G$. This leads to a contradiction and completes the proof for the projective-planar case of Ding and Marshall's conjecture.

\section{Spanning subdivisions in graphs without specific minors}\label{sec:DEEF}

In this section, as part of our preparation for the main proofs, we investigate the structure of graphs that exclude certain minors. We establish the existence of specific spanning subdivisions, which we later use to derive a contradiction concerning the existence of a Hamilton cycle.

Although we intend to analyze counterexamples, we do not use their non-hamiltonicity in this section; such arguments are deferred to later sections.

We consider the four graphs \( D_{17} \), \( E_{20} \), \( E_{22} \), and \( F_4 \), shown in Figure~\ref{fig:DEEF}, with the vertex labelings as depicted. These graphs are internally $4$-connected and play a central role in this section.

\begin{figure}[!ht]
	\centering
	\begin{tabular}{cccc}
		\subfloat[The graph $D_{17}$.\label{subfig:D17}]{
			\includegraphics[scale=1]{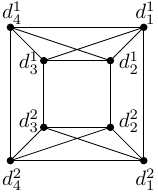}
		} & 
		\subfloat[The graph $E_{20}$.\label{subfig:E20}]{
			\includegraphics[scale=1]{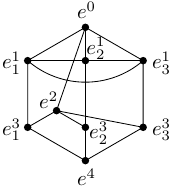}
		} & 
		\subfloat[The graph $E_{22}$.\label{subfig:E22}]{
			\includegraphics[scale=1]{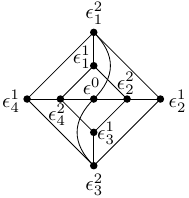}
		} & 
		\subfloat[The graph $F_4$.\label{subfig:F4}]{
			\includegraphics[scale=1]{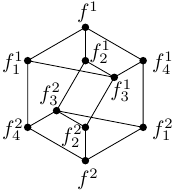}
		} 
	\end{tabular}
	\caption{Four internally 4-connected non-projective-planar graphs.}
	\label{fig:DEEF}
\end{figure}

The symmetries of these graphs help reduce the case analysis. 

The graph \( D_{17} \) has 48 automorphisms, each either fixing \( \{d^1_1, d^1_2, d^1_3, d^1_4\} \) or mapping it to \( \{d^2_1, d^2_2, d^2_3, d^2_4\} \). Those that fix \( \{d^1_1, d^1_2, d^1_3, d^1_4\} \) correspond one-to-one with the permutations of these four vertices. 

The graph \( E_{20} \) has 6 automorphisms, corresponding one-to-one with the permutations of \( \{e^1_1, e^1_2, e^1_3\} \). 

The graph \( E_{22} \) has 24 automorphisms, corresponding one-to-one with the permutations of \( \{\epsilon^1_1, \epsilon^1_2, \epsilon^1_3, \epsilon^1_4\} \). 

The graph \( F_4 \) has 4 automorphisms, each mapping \( f^1_1 \) to one of \( f^1_1 \), \( f^1_4 \), \( f^2_1 \), or \( f^2_4 \).

We first prove a proposition concerning graphs that contain a $D_{17}$ minor.

\begin{proposition}\label{pro:D17toE20}
	Let $G$ be a $3$-connected graph containing a minor of $D_{17}$ with $|V(G)| > |V(D_{17})|$. Then $G$ has a minor of $K_{3,4}$ or $E_{20}$.
\end{proposition}
\begin{proof}
By Theorem~\ref{thm:Seymour}, there exists a graph \( H \) such that \( G \) contains \( H \) as a minor, and \( H \) is obtained from \( D_{17} \) by adding some edges (possibly none) and then splitting one vertex.

Without loss of generality, suppose the vertex \( d^1_1 \) is split into \( v_1 \) and \( v_2 \) to obtain \( H \).

If, say, \( v_1 \) has neighbors \( d^1_2, d^1_3, d^1_4 \), then \( v_2 \) must have two neighbors from \( \{d^2_1, d^2_2, d^2_3, d^2_4\} \). Depending on whether \( d^2_1 \) is a neighbor of \( v_1 \), \( H \) contains the first or second graph in Figure~\ref{fig:D17toE20} as a subgraph. Thus, \( G \) contains \( K_{3,4} \) or \( E_{20} \) as a minor.

If, say, \( v_1 \) has neighbors \( d^1_2, d^1_3 \) and \( v_2 \) has \( d^1_4 \) as a neighbor, then \( H \) contains the third, fourth, or fifth graph in Figure~\ref{fig:D17toE20} as a subgraph. Consequently, \( G \) contains \( E_{20} \) or \( K_{3,4} \) as a minor.
\end{proof}

\begin{figure}[!ht]
	\centering{%
		\includegraphics[scale=1]{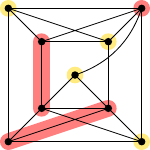}
	}
	\hspace{10pt}
	{%
		\includegraphics[scale=1]{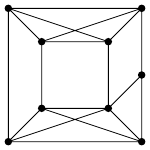}
	}
	\hspace{10pt}
	{%
		\includegraphics[scale=1]{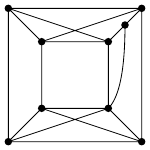}
	}
	\hspace{10pt}
	{%
		\includegraphics[scale=1]{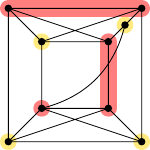}
	}
	\hspace{10pt}
	{%
		\includegraphics[scale=1]{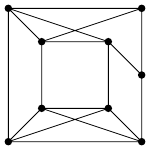}
	}
	\caption{The first and fourth graphs contains \( K_{3,4} \) as a minor. The other three graphs contain \( E_{20} \) as a subgraph. (In fact, the first graph can be obtained from the fourth graph by adding an edge.)}
	\label{fig:D17toE20}
\end{figure}

Next, we discuss graphs containing \( E_{20} \) as a minor and prove two propositions, one for Ding and Marshall's conjecture and the other for their question.

\begin{lemma} \label{lem:E20sub}
	Let $H$ be a graph obtained from \( E_{20} \) by splitting a vertex. Then $H$ contains a minor of $K_{3,4}$ or $F_4$.
\end{lemma}
\begin{proof}
	Up to symmetry, the graphs obtained from \( E_{20} \) by splitting a vertex are the first four graphs shown in Figure~\ref{fig:E20sub}. Each of these graphs contains a minor of $F_4$ or \( K_{3,4} \).
\end{proof}

\begin{figure}[!ht]
	\centering{%
		\includegraphics[scale=1]{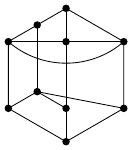}
	}
	\hfill
	{%
		\includegraphics[scale=1]{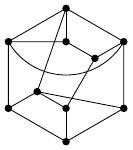}
	}
	\hfill
	{%
		\includegraphics[scale=1]{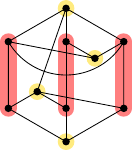}
	}
	\hfill
	{%
		\includegraphics[scale=1]{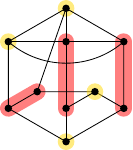}
	}
	\hfill
	{%
		\includegraphics[scale=1]{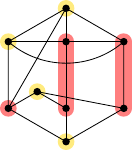}
	}
	\hfill
	{%
		\includegraphics[scale=1]{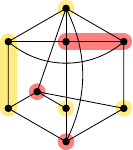}
	}
	\caption{The first two graphs contain $F_4$ as a subgraph. The other graphs contain $K_{3,4}$ as a minor.}
	\label{fig:E20sub}
\end{figure}

\begin{proposition}\label{pro:E20toF4DM}
	Let \( G \) be a \( 3 \)-connected graph. Suppose \( G \) contains a minor of \( E_{20} \) and has no two adjacent vertices each with degree at least \( 4 \). Then \( G \) contains a minor of \( K_{3,4} \), \( \mathfrak{Q}^+ \), or \( F_4 \).
\end{proposition}
\begin{proof}
	Suppose, for contradiction, that \( G \) has no minor of \( K_{3,4} \), \( \mathfrak{Q}^+ \), or \( F_4 \).
	
	By Lemma~\ref{lem:E20sub}, \( G \) contains a subdivision of \( E_{20} \). Let \( \eta \) denote such a subdivision, that is, \( \eta \) is a mapping on \( V(E_{20}) \cup E(E_{20}) \) such that \( \eta \) maps the vertices of \( E_{20} \) to distinct vertices of \( G \) and maps the edges of \( E_{20} \) to internally disjoint paths in \( G \), satisfying the condition that for any \( uv \in E(E_{20}) \), the path \( \eta(uv) \) has end-vertices \( \eta(u) \) and \( \eta(v) \) and contains no other vertices from \( \eta(V(E_{20})) \).
	
	Since \( G \) is 3-connected and has no adjacent vertices each with degree at least 4, the path \( \eta(e^1_1 e^1_3) \) must have an internal vertex that connects to a vertex outside \( \eta(e^1_1 e^1_3) \) via a path. This implies that \( G \) contains one of the graphs shown in Figure~\ref{fig:E20toF4e} as a minor. These graphs contain \( F_4 \), \( K_{3,4} \), or \( \mathfrak{Q}^+ \) as minors, leading to a contradiction.
\end{proof}

\begin{figure}[!ht]
	\centering{%
		\includegraphics[scale=1]{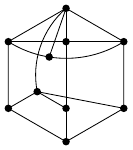}
	}
	\hfill
	{%
		\includegraphics[scale=1]{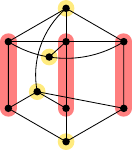}
	}
	\hfill
	{%
		\includegraphics[scale=1]{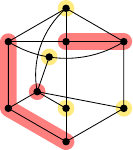}
	}
	\hfill
	{%
		\includegraphics[scale=1]{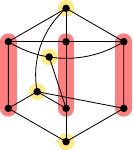}
	}
	\hfill
	{%
		\includegraphics[scale=1]{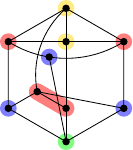}
	}
	\hfill
	{%
		\includegraphics[scale=1]{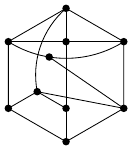}
	}
	\caption{The first and sixth graphs contain \( F_4 \) as a subgraph, the fifth graph contains a minor of \( \mathfrak{Q}^+ \), and the remaining graphs contain \( K_{3,4} \) as a minor.}
	\label{fig:E20toF4e}
\end{figure}

\begin{lemma} \label{lem:E20+I}
	Let \( H \) be a graph obtained from \( E_{20} \) by adding an edge joining one of the following pairs: \( \{e^0, e^3_1\} \), \( \{e^0, e^3_2\} \), \( \{e^0, e^3_3\} \), \( \{e^0, e^4\} \). Then \( H \) contains \( K_{3,4} \) as a minor.
\end{lemma}
\begin{proof}
	Up to symmetry, \( H \) is either the fifth or sixth graph from Figure~\ref{fig:E20sub}. In each case, it is not hard to see that \( H \) contains \( K_{3,4} \) as a minor.
\end{proof}

\begin{lemma} \label{lem:E20+Y}
	Let \( H \) be a graph obtained from \( E_{20} \) by adding a new vertex \( v \) and joining it to three vertices. Let \( S \) be the set of neighbors of \( v \). If every pair of vertices from \( S \) are adjacent or every pair of vertices from $S$ are non-adjacent, then $H$ contains \( K_{3,4} \) or $F_4$ as a minor.
\end{lemma}
\begin{proof}
	Up to symmetry, there are two possibilities that every pair of vertices from $S$ are adjacent: $S$ is $\{e^0, e^1_1, e^1_2\}$ or $\{e^1_1, e^1_2, e^1_3\}$. In either case, $H$ contains the first or third graph in Figure~\ref{fig:E20sub} as a subgraph. Thus, $H$ contains $F_4$ or $K_{3,4}$ as a minor.
	
	By Lemma~\ref{lem:E20+I}, if $S$ contains $e^0$, then it contains none of $e^3_1, e^3_2, e^3_3, e^4$. Therefore, up to symmetry, if every pair of vertices from $S$ are non-adjacent, then $S$ is one of the following sets: $\{e^1_1, e^2, e^4\}$, $\{e^1_1, e^3_2, e^3_3\}$, $\{e^3_1, e^3_2, e^3_3\}$. They are depicted as the first three graphs in Figure~\ref{fig:E20+Y}. Each of them contains a \( K_{3,4} \) minor.
\end{proof}

\begin{figure}[!ht]
	\centering{%
		\includegraphics[scale=1]{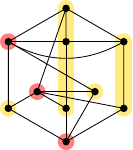}
	}
	\hspace{10pt}
	{%
		\includegraphics[scale=1]{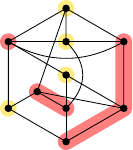}
	}
	\hspace{10pt}
	{%
		\includegraphics[scale=1]{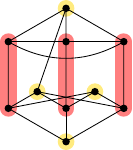}
	}
	\hspace{10pt}
	{%
		\includegraphics[scale=1]{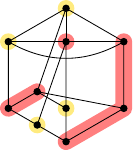}
	}
	\hspace{10pt}
	{%
		\includegraphics[scale=1]{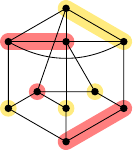}
	}
	\caption{Five graphs, each containing \( K_{3,4} \) as a minor.}
	\label{fig:E20+Y}
\end{figure}

\begin{lemma} \label{lem:E20+X}
	Let \( H \) be a graph obtained from \( E_{20} \) by adding a new vertex and joining it to four vertices of \( E_{20} \). Then \( H \) contains \( K_{3,4} \) or $F_4$ as a minor.
\end{lemma}
\begin{proof}
	Denote by \( S \) the set of neighbors of the newly added vertex. Similarly to before, we can apply Lemmas~\ref{lem:E20+I} and~\ref{lem:E20+Y} to exclude certain possibilities for \( S \). Thus, up to symmetry, we only need to consider the cases where \( S \) is one of the following sets: 
	\(\{e^1_1, e^1_2, e^2, e^3_1\}\), 
	\(\{e^1_1, e^1_2, e^2, e^3_3\}\), 
	\(\{e^1_1, e^1_2, e^3_1, e^3_2\}\), 
	\(\{e^1_1, e^1_2, e^3_1, e^4\}\), 
	\(\{e^1_1, e^1_2, e^3_3, e^4\}\), 
	\(\{e^1_1, e^2, e^3_1, e^3_2\}\), 
	\(\{e^1_1, e^3_1, e^3_2, e^4\}\), 
	\(\{e^2, e^3_1, e^3_2, e^4\}\).
	
	If \( S \) is \( \{e^1_1, e^1_2, e^2, e^3_1\} \) or \( \{e^1_1, e^1_2, e^2, e^3_3\} \), then \( H \) contains the third graph in Figure~\ref{fig:E20toF4e} as a subgraph.

	If \( S = \{e^1_1, e^1_2, e^3_1, e^3_2\} \), then \( H \) contains the sixth graph in Figure~\ref{fig:E20toF4e} as a subgraph.
	
	If \( S \) is \( \{e^1_1, e^1_2, e^3_1, e^4\} \), \( \{e^1_1, e^1_2, e^3_3, e^4\} \), or \( \{e^1_1, e^3_1, e^3_2, e^4\} \), then \( H \) contains the fourth graph in Figure~\ref{fig:E20+Y} as a subgraph.
	
	If \( S = \{e^1_1, e^2, e^3_1, e^3_2\} \), then \( H \) contains the fifth graph in Figure~\ref{fig:E20+Y} as a subgraph.

	If \( S = \{e^2, e^3_1, e^3_2, e^4\} \), then \( H \) contains the fourth graph in Figure~\ref{fig:E20sub} as a subgraph.
	
	In every case, \( H \) contains a minor of \( K_{3,4} \) or \( F_4 \).
\end{proof}

\begin{lemma} \label{lem:E20+H}
	Let \( e_1 \) and \( e_2 \) be two independent edges of \( E_{20} \). Let \( H \) be the graph obtained from \( E_{20} \) by subdividing each of \( e_1 \) and \( e_2 \) with a new vertex and then adding an edge between these two new vertices. Then \( H \) contains \( K_{3,4} \) or $F_4$ as a minor.
\end{lemma}\begin{proof}
Again, we apply Lemma~\ref{lem:E20+I} to exclude certain possibilities for \( \{e_1, e_2\} \). Up to symmetry, \( \{e_1, e_2\} \) is one of the following sets: 
\(\{e^0 e^1_1, e^1_2 e^1_3\}\), 
\(\{e^0 e^2, e^1_1 e^1_2\}\), 
\(\{e^1_1 e^1_2, e^1_3 e^3_3\}\), 
\(\{e^1_1 e^1_2, e^2 e^3_1\}\), 
\(\{e^1_1 e^1_2, e^2 e^3_3\}\), 
\(\{e^1_1 e^1_2, e^3_1 e^4\}\), 
\(\{e^1_1 e^1_2, e^3_3 e^4\}\), 
\(\{e^1_1 e^3_1, e^1_2 e^3_2\}\), 
\(\{e^1_1 e^3_1, e^2 e^3_2\}\), 
\(\{e^1_1 e^3_1, e^3_2 e^4\}\), 
\(\{e^2 e^3_1, e^3_2 e^4\}\).

If \( \{e_1, e_2\} \) is \( \{e^0 e^1_1, e^1_2 e^1_3\} \) or \( \{e^0 e^2, e^1_1 e^1_2\} \), then \( H \) contains the first graph in Figure~\ref{fig:E20toF4e} as a minor.

If \( \{e_1, e_2\} \) is \( \{e^1_1 e^1_2, e^1_3 e^3_3\} \), \( \{e^1_1 e^1_2, e^2 e^3_3\} \), or \( \{e^1_1 e^1_2, e^3_3 e^4\} \), then \( H \) contains the fourth graph in Figure~\ref{fig:E20toF4e} as a minor.

If \( \{e_1, e_2\} \) is \( \{e^1_1 e^1_2, e^2 e^3_1\} \) or \( \{e^1_1 e^1_2, e^3_1 e^4\} \), then \( H \) contains the sixth graph in Figure~\ref{fig:E20toF4e} as a minor.

If \( \{e_1, e_2\} = \{e^1_1 e^3_1, e^1_2 e^3_2\} \), then one can obtain \( F_4 \) from \( H \) by removing \( e^0 e^1_3 \) and \( e^1_1 e^1_2 \) and contracting \( e^1_1 v \), where \( v \) is the vertex subdividing \( e^1_1 e^3_1 \).

If \( \{e_1, e_2\} \) is \( \{e^1_1 e^3_1, e^2 e^3_2\} \), \( \{e^1_1 e^3_1, e^3_2 e^4\} \), or \( \{e^2 e^3_1, e^3_2 e^4\} \), then \( H \) contains the fourth graph in Figure~\ref{fig:E20sub} as a minor. (In fact, the graphs \( H \) corresponding to \( \{e_1, e_2\} = \{e^1_1 e^3_1, e^3_2 e^4\} \) and \( \{e_1, e_2\} = \{e^2 e^3_1, e^3_2 e^4\} \) are isomorphic.)

In every case, \( H \) contains a minor of \( K_{3,4} \) or \( F_4 \).
\end{proof}

\begin{proposition} \label{pro:E20span}
	Let $G$ be a $4$-connected graph that has neither a minor of $K_{3,4}$ nor a minor of $F_4$. Then every JT-subdivision $\eta(E_{20})$ of $E_{20}$ in $G$ is a spanning subgraph of $G$.
\end{proposition}

\begin{proof}	
	Suppose, to the contrary, that $\eta(E_{20})$ does not span $G$. So, there exists a bridge $B$ of $\eta(E_{20})$. 
	
Since \( G \) and \( E_{20} \) are internally 4-connected and Lemma~\ref{lem:E20sub} holds, we conclude by Lemma~\ref{lem:JTstable} that \( B \) is a stable bridge.

	We claim that no two attachments of \( B \) are contained in the same segment. Suppose otherwise, and let there be a segment \( \eta(e) \) containing two attachments of \( B \). By Lemma~\ref{lem:E20+H}, if \( B \) has an attachment in a segment disjoint from \( \eta(e) \), then that attachment must be in \( \eta(V(E_{20})) \). Therefore, by the stability of \( B \), it follows that \( G \) contains a minor of a graph obtained from \( E_{20} \) by adding a vertex and joining it either to four vertices of \( E_{20} \) or to three pairwise adjacent vertices of \( E_{20} \). By Lemma~\ref{lem:E20+X} or Lemma~\ref{lem:E20+Y}, this implies that \( G \) contains a minor of \( K_{3,4} \) or \( F_4 \), contradicting our assumption. Hence, no segment contains two attachments of \( B \).
	
	For any attachment \( v \) of \( B \) in \( \eta(u_1 u_2) - \eta(u_1) - \eta(u_2) \) for some \( u_1 u_2 \in E(E_{20}) \), we can move \( v \) to \( \eta(u_1) \) or \( \eta(u_2) \) by contracting \( \eta(u_1 u_2)[\eta(u_1), v] \) or \( \eta(u_1 u_2)[v, \eta(u_2)] \). When a vertex is moved, the graph is modified; however, this does not affect our argument, as our goal is to derive a contradiction by finding a minor of \( K_{3,4} \) or \( F_4 \).

	Let \( v_1, v_2, v_3, v_4 \) be four attachments of \( B \), which exist because \( G \) is 4-connected. We show that we can move these attachments step by step to distinct vertices in \( \eta(V(E_{20})) \). This will then lead to a contradiction by Lemma~\ref{lem:E20+X}.
	
	An attachment \( v_i \) with \( i \in [4] \) is \emph{active} if it is not in \( \eta(V(E_{20})) \); otherwise, it is \emph{inactive}. We move all active attachments to make them inactive while ensuring that no two attachments reach the same vertex in \( \eta(V(E_{20})) \). An active attachment \( v_i \) in \( \eta(u_1 u_2) - \eta(u_1) - \eta(u_2) \) for some \( u_1 u_2 \in E(E_{20}) \) is \emph{spooked} if another attachment has already been moved to \( \eta(u_1) \) or \( \eta(u_2) \). A vertex in \( \eta(V(E_{20})) \) is \emph{available} if none of \( v_1, v_2, v_3, v_4 \) occupies it.
	
While there are active attachments, we take a step of moves as follows. If there are spooked attachments, move all of them to available vertices. Otherwise, take one active (but not spooked) attachment and move it to an available vertex. 

It remains to show that any spooked attachment can be moved to an available vertex in \( \eta(V(E_{20})) \) and that when multiple spooked attachments need to be moved in a single step, they are placed at distinct available vertices.

	If at some step there is no available vertex for a spooked attachment, say \( v_4 \), in \( \eta(u_1 u_2) - \eta(u_1) - \eta(u_2) \) for some \( u_1 u_2 \in E(E_{20}) \), then two other attachments, say \( v_3 \) and \( v_2 \), must have already been moved to \( \eta(u_1) \) and \( \eta(u_2) \) in the previous step. Since only one vertex, say \( v_1 \), was moved in the first step, the sequence of events is as follows: after moving \( v_1 \) in the first step, the attachments \( v_2 \) and \( v_3 \) became spooked and were subsequently moved to \( \eta(u_1) \) and \( \eta(u_2) \) in the second step. Consequently, the segments containing \( v_2, v_3, v_4 \) form the image of a cycle of length three under \( \eta \). By Lemma~\ref{lem:E20+Y}, \( G \) contains a minor of \( F_4 \) or \( K_{3,4} \), which is a contradiction. These arguments also show that when multiple spooked attachments need to be moved in a single step, they can always be assigned to distinct available vertices in \( \eta(V(E_{20})) \). This completes the proof.
\end{proof}

We next prove a proposition about graphs with \( E_{22} \) minors.

\begin{proposition}\label{pro:E22}
	Let $G$ be a $4$-connected graph. If $G$ contains a minor of $E_{22}$, then it also contains a minor of $K_{3,4}$.
\end{proposition}

\begin{proof}
Suppose, to the contrary, that \( G \) has no \( K_{3,4} \) minor. 

As \( G \) contains \( E_{22} \) as a minor, we have a mapping \( \mu \) that maps \( V(E_{22}) \) into pairwise disjoint subsets of \( V(G) \) and maps \( E(E_{22}) \) to pairwise internally disjoint paths in \( G \), such that for each \( v \in V(E_{22}) \), the induced subgraph \( G[\mu(v)] \) is connected, and for any \( uv \in E(E_{22}) \), the path \( \mu(uv) \) joins \( \mu(u) \) and \( \mu(v) \) with no internal vertex of \( \mu(uv) \) in \( \bigcup_{w \in V(E_{22})} \mu(w) \). 

We may further require that for \( v \in \{\epsilon^1_1, \epsilon^1_2, \epsilon^1_3, \epsilon^1_4\} \), \( |\mu(v)| = 1 \), and for \( v \in \{\epsilon^0, \epsilon^2_1, \epsilon^2_2, \epsilon^2_3, \epsilon^2_4\} \), if \( |\mu(v)| > 1 \), then \( G[\mu(v)] \) has a spanning path such that each of its end-vertices is incident to precisely two paths \( \mu(e) \) corresponding to the edges \( e \) incident to \( v \). We say that \( v \in V(E_{22}) \) is \emph{split} if \( |\mu(v)|>1 \).

Note that at most one of \( \epsilon^2_1, \epsilon^2_2, \epsilon^2_3, \epsilon^2_4 \) is split; otherwise, \( G \) would contain one of the first four graphs given in Figure~\ref{fig:E222s} and hence \( K_{3,4} \) as a minor, which is a contradiction. 

Suppose one of \( \epsilon^2_1, \epsilon^2_2, \epsilon^2_3, \epsilon^2_4 \) is split, say \( \epsilon^2_3 \). Then \( \epsilon^2_1 \), \( \epsilon^2_2 \), and \( \epsilon^2_4 \) are not split. Since \( \mu(\epsilon^2_1) \cup \mu(\epsilon^2_2) \cup \mu(\epsilon^2_4) \) is not a 3-cut of \( G \), there exists a path \( P \) internally disjoint from \( \mu(E_{22}) \) with one end-vertex in \( (\mu(\epsilon^1_1 \epsilon^2_1) \cup \mu(\epsilon^1_1 \epsilon^2_2) \cup \mu(\epsilon^1_1 \epsilon^2_4)) - (\mu(\epsilon^2_1) \cup \mu(\epsilon^2_2) \cup \mu(\epsilon^2_4)) \) and the other end-vertex in \( \mu(E_{22}) - (\mu(\epsilon^1_1 \epsilon^2_1) \cup \mu(\epsilon^1_1 \epsilon^2_2) \cup \mu(\epsilon^1_1 \epsilon^2_4))  \), where \( \mu(E_{22}) \) denotes the subgraph of \( G \) induced by the union of the images of \( \mu \) over the vertices and edges of \( E_{22} \). Moreover, \( P \) has no end-vertex in \( \mu(\epsilon^2_3) \), since otherwise it would contain the fifth graph in Figure~\ref{fig:E222s} and hence \( K_{3,4} \) as a minor. We observe that the graph obtained from \( E_{22} \) by splitting \( \epsilon^1_1 \) is unique up to symmetry and admits two automorphisms. Exploiting this symmetry, we deduce that \( G \) must contain one of the first three graphs depicted in Figure~\ref{fig:E221s}, and thus \( K_{3,4} \) as a minor, which leads to a contradiction.

Therefore, none of \( \epsilon^2_1, \epsilon^2_2, \epsilon^2_3, \epsilon^2_4 \) is split. 

Now, for each \( i \in [1] \), there exists a path \( P_i \) internally disjoint from \( \mu(E_{22}) \) with one end-vertex in \( (\mu(\epsilon^1_i \epsilon^2_i) \cup \mu(\epsilon^1_i \epsilon^2_{i+1}) \cup \mu(\epsilon^1_i \epsilon^2_{i+3})) - (\mu(\epsilon^2_i) \cup \mu(\epsilon^2_{i+1}) \cup \mu(\epsilon^2_{i+3})) \) and the other end-vertex in \( \mu(E_{22}) - (\mu(\epsilon^1_i \epsilon^2_i) \cup \mu(\epsilon^1_i \epsilon^2_{i+1}) \cup \mu(\epsilon^1_i \epsilon^2_{i+3}))  \), with indices taken modulo \( 4 \). Again, \( P_i \) has no end-vertex in \( \mu(\epsilon^2_{i+2}) \), since otherwise it would contain the fifth graph in Figure~\ref{fig:E222s} as a minor. 

Considering the union of \( \mu(E_{22}) \) and the paths \( P_1, P_2, P_3, P_4 \), it is easy to see that \( G \) must contain a minor of one of the last three graphs depicted in Figure~\ref{fig:E221s}. This leads to the contradiction that \( G \) has a \( K_{3,4} \) minor.
\end{proof}

\begin{figure}[!ht]
	\centering{%
		\includegraphics[scale=1]{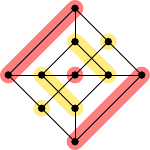}
	}
	\hspace{10pt}
	{%
		\includegraphics[scale=1]{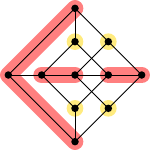}
	}
	\hspace{10pt}
	{%
		\includegraphics[scale=1]{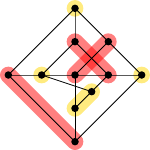}
	}
	\hspace{10pt}
	{%
		\includegraphics[scale=1]{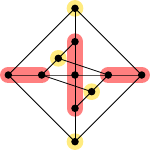}
	}
	\hspace{10pt}
	{%
		\includegraphics[scale=1]{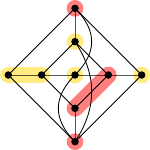}
	}
	\caption{Five graphs, each containing a minor of $K_{3,4}$.}
	\label{fig:E222s}
\end{figure}

\begin{figure}[!ht]
	\centering{%
		\includegraphics[scale=1]{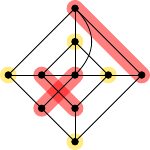}
	}
	\hfill
	{%
		\includegraphics[scale=1]{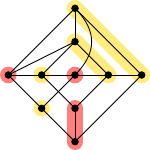}
	}
	\hfill
	{%
		\includegraphics[scale=1]{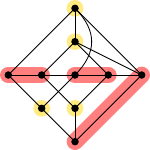}
	}
	\hfill
	{%
		\includegraphics[scale=1]{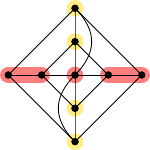}
	}
	\hfill
	{%
		\includegraphics[scale=1]{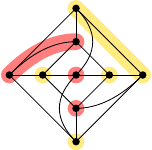}
	}
	\hfill
	{%
		\includegraphics[scale=1]{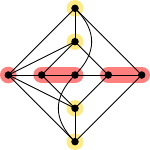}
	}
	\caption{Six graphs, each containing a minor of $K_{3,4}$.}
	\label{fig:E221s}
\end{figure}

Finally, we turn our attention to graphs containing \( F_4 \) minors. In analogy with our treatment of \( E_{20} \), we establish a sequence of lemmas leading to a concluding proposition.

\begin{lemma} \label{lem:F4sub}
	Let $H$ be a graph obtained from \( F_4 \) by splitting a vertex. Then $H$ contains a minor of $K_{3,4}$.
\end{lemma}
\begin{proof}
	Up to symmetry, the graphs obtained from \( F_4 \) by splitting a vertex are the first two graphs shown in Figure~\ref{fig:F4sub}. Each of these graphs contains a \( K_{3,4} \) minor.
\end{proof}

\begin{figure}[!ht]
	\centering{%
		\includegraphics[scale=1]{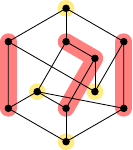}
	}
	\hfill
	{%
		\includegraphics[scale=1]{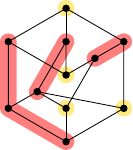}
	}
	\hfill
	{%
		\includegraphics[scale=1]{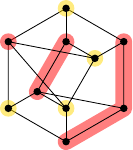}
	}
	\hfill
	{%
		\includegraphics[scale=1]{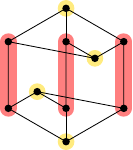}
	}
	\hfill
	{%
		\includegraphics[scale=1]{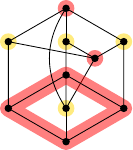}
	}
	\hfill
	{%
		\includegraphics[scale=1]{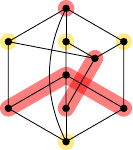}
	}
	\caption{Six graphs, each containing \( K_{3,4} \) as a minor.}
	\label{fig:F4sub}
\end{figure}

\begin{lemma} \label{lem:F4+I}
	Let \( H \) be a graph obtained from \( F_4 \) by adding an edge joining one of the following pairs: \( \{f^1_1, f^2_2\} \), \( \{f^1_2, f^2_4\} \), \( \{f^1_2, f^2_1\} \), \( \{f^1_4, f^2_2\} \), \( \{f^1_2, f^2_2\} \), \( \{f^1, f^2_2\} \), \( \{f^1_2, f^2\} \), \( \{f^1, f^2\} \). Then \( H \) contains \( K_{3,4} \) as a minor.
\end{lemma}
\begin{proof}
	Up to symmetry, \( H \) contains either the third, fourth, fifth, or sixth graph from Figure~\ref{fig:F4sub} as a subgraph. In each case, it is clear that \( H \) contains \( K_{3,4} \) as a minor.
\end{proof}

\begin{lemma} \label{lem:F4+Y}
	Let \( H \) be a graph obtained from \( F_4 \) by adding a new vertex \( v \) and joining it to three vertices of \( F_4 \). Let \( S \) be the set of neighbors of \( v \). If \( H \) does not contain \( K_{3,4} \) as a minor, then either $S$ contains two adjacent vertices in \( H \), or \( S = \{f^1_1, f^1_4, f^2\} \), or \( S = \{f^2_1, f^2_4, f^1\} \).  
	Moreover, when \( S = \{f^1_1, f^1_4, f^2\} \) or \( S = \{f^2_1, f^2_4, f^1\} \), the graph \( H \) contains \( \mathfrak{Q}^+ \) as a minor.  
\end{lemma}

\begin{proof}
	As illustrated in the first graph in Figure~\ref{fig:F4+Y}, \( H \) contains \( \mathfrak{Q}^+ \) as a minor when \( S \) is \( \{f^1_1, f^1_4, f^2\} \) or \( \{f^2_1, f^2_4, f^1\} \).  
	
	Now, suppose that \( S \) is an independent set and is neither \( \{f^1_1, f^1_4, f^2\} \) nor \( \{f^2_1, f^2_4, f^1\} \). We will show that \( H \) contains a \( K_{3,4} \) minor.  
	
	By Lemma~\ref{lem:F4+I}, the set \( S \) cannot contain any of the following pairs:  
	\( \{f^1_1, f^2_2\} \),  
	\( \{f^1_2, f^2_4\} \),  
	\( \{f^1_2, f^2_1\} \),  
	\( \{f^1_4, f^2_2\} \),  
	\( \{f^1_2, f^2_2\} \),  
	\( \{f^1, f^2_2\} \),  
	\( \{f^1_2, f^2\} \), or  
	\( \{f^1, f^2\} \).  
	
	Therefore, up to symmetry, \( S \) must be one of the following sets:  
	\( \{f^1_1, f^1_2, f^1_4\} \),  
	\( \{f^1, f^1_3, f^2_3\} \),  
	\( \{f^1, f^1_3, f^2_1\} \), or  
	\( \{f^1_1, f^1_4, f^2_3\} \). 
	The corresponding graphs \( H \) are depicted as the last four graphs in Figure~\ref{fig:F4+Y}. Each of these graphs contains a minor of \( K_{3,4} \).
\end{proof}

\begin{figure}[!ht]
	\centering{%
		\includegraphics[scale=1]{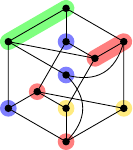}
	}
	\hspace{10pt}
	{%
		\includegraphics[scale=1]{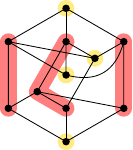}
	}
	\hspace{10pt}
	{%
		\includegraphics[scale=1]{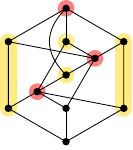}
	}
	\hspace{10pt}
	{%
		\includegraphics[scale=1]{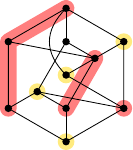}
	}
	\hspace{10pt}
	{%
		\includegraphics[scale=1]{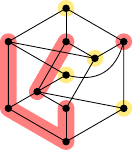}
	}
	\caption{The first graph contains \( \mathfrak{Q}^+ \) as a minor. The last four graphs contain \( K_{3,4} \) as a minor.}
	\label{fig:F4+Y}
\end{figure}

\begin{lemma} \label{lem:F4+X}
	Let \( H \) be a graph obtained from \( F_4 \) by adding a new vertex and joining it to four vertices of \( F_4 \). Then \( H \) contains \( K_{3,4} \) as a minor.
\end{lemma}
\begin{proof}
	Let $v$ be the new vertex, and \( S \) be the set of neighbors of \( v \).
	Suppose, to the contrary, that $H$ does not contain $K_{3,4}$ as a minor.
	
	Suppose \( S \supset \{f^1_1, f^1_4, f^2\} \). If \( S \) contains one of \( f^1, f^1_2, f^2_2 \) (as the fourth neighbor of $v$), then by Lemma~\ref{lem:F4+I}, \( H \) contains a \( K_{3,4} \) minor. If \( S \) contains \( f^1_3 \) or \( f^2_3 \), then \( H \) contains the first graph in Figure~\ref{fig:F4sub} or the fifth graph in Figure~\ref{fig:F4+Y}, and hence a \( K_{3,4} \) minor. If \( S \) contains \( f^2_1 \) or \( f^2_4 \), then \( H \) contains a minor of the second graph in Figure~\ref{fig:F4+X}, and thus a \( K_{3,4} \) minor. Thus, \( S \not\supset \{f^1_1, f^1_4, f^2\} \), and, by symmetry, \( S \not\supset \{f^2_1, f^2_4, f^1\} \).

	We show that it suffices to assume \( S \) consists of two disjoint pairs of adjacent vertices. By Lemma~\ref{lem:F4+Y}, \( S \) contains two adjacent vertices, say \( v_1 \) and \( v_2 \). If \( v_3 \) and \( v_4 \) are not adjacent, applying Lemma~\ref{lem:F4+Y} and the fact that \( F_4 \) has no cycle of length 3, we conclude that, without loss of generality, \( v_1 \) is adjacent to \( v_3 \) and \( v_2 \) is adjacent to \( v_4 \). Hence, we can assume that \( S \) consists of two pairs of adjacent vertices.
	
	By Lemma~\ref{lem:F4+I}, \( S \) cannot contain any of the following pairs: \( \{f^1_1, f^2_2\} \), \( \{f^1_2, f^2_4\} \), \( \{f^1_2, f^2_1\} \), \( \{f^1_4, f^2_2\} \), \( \{f^1_2, f^2_2\} \), \( \{f^1, f^2_2\} \), \( \{f^1_2, f^2\} \), or \( \{f^1, f^2\} \).
	
	Therefore, up to symmetry, \( S \) is one of the following sets:  
	\( \{f^1, f^1_1, f^1_2, f^1_3\} \),  
	\( \{f^1, f^1_1, f^1_2, f^2_3\} \),  
	\( \{f^1, f^1_1, f^1_3, f^1_4\} \),  
	\( \{f^1, f^1_1, f^1_4, f^2_1\} \),  
	\( \{f^1, f^1_1, f^2_1, f^2_3\} \),  
	\( \{f^1, f^1_1, f^2_3, f^2_4\} \),  
	\( \{f^1_1, f^1_2, f^1_3, f^2_3\} \),  
	\( \{f^1_1, f^1_3, f^1_4, f^2_1\} \),  
	\( \{f^1_1, f^1_3, f^2_1, f^2_3\} \),  
	\( \{f^1_1, f^1_3, f^2_3, f^2_4\} \),  
	or \( \{f^1_1, f^1_4, f^2_1, f^2_4\} \).
	
	If \( S \) is \( \{f^1, f^1_1, f^1_2, f^1_3\} \) or \( \{f^1_1, f^1_2, f^1_3, f^2_3\} \), then $S \supset \{f^1_1, f^1_2, f^1_3\}$ and \( H \) contains the second graph in Figure~\ref{fig:F4sub} as a subgraph. 
	
	If \( S \) is \( \{f^1, f^1_1, f^1_2, f^2_3\} \), or \(  \{f^1, f^1_1, f^2_1, f^2_3\} \), or \( \{f^1, f^1_1, f^2_3, f^2_4\} \), then $S \supset \{f^1, f^1_1, f^2_3\}$ and \( H \) contains the first graph in Figure~\ref{fig:F4+X} as a subgraph. 
	
	If \( S \) is \( \{f^1, f^1_1, f^1_3, f^1_4\} \) or \( \{f^1_1, f^1_3, f^1_4, f^2_1\} \), then $S \supset \{f^1_1, f^1_3, f^1_4\}$ and \( H \) contains the first graph in Figure~\ref{fig:F4sub} as a subgraph.  
	
	If \( S = \{f^1, f^1_1, f^1_4, f^2_1\} \), then \( H \) contains the second graph in Figure~\ref{fig:F4+X} as a subgraph.  
	
	If \( S = \{f^1_1, f^1_3, f^2_1, f^2_3\} \) or \( S = \{f^1_1, f^1_3, f^2_3, f^2_4\} \), then $S \supset \{f^1_1, f^1_3, f^2_3\}$ and \( H \) contains the third graph in Figure~\ref{fig:F4+X} as a subgraph. 
	
	If \( S = \{f^1_1, f^1_4, f^2_1, f^2_4\} \), then \( H \) contains the fourth graph in Figure~\ref{fig:F4+X} as a subgraph. 
	
	In all cases, \( H \) contains a \( K_{3,4} \) minor.
\end{proof}

\begin{figure}[!ht]
	\centering{%
		\includegraphics[scale=1]{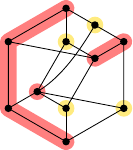}
	}
	\hfill
	{%
		\includegraphics[scale=1]{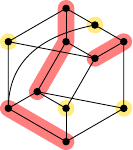}
	}
	\hfill
	{%
		\includegraphics[scale=1]{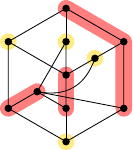}
	}
	\hfill
	{%
		\includegraphics[scale=1]{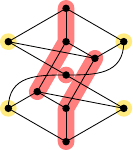}
	}
	\hfill
	{%
		\includegraphics[scale=1]{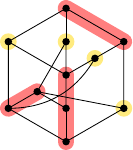}
	}
	\hfill
	{%
		\includegraphics[scale=1]{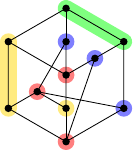}
	}
	\caption{The first five graphs contain \( K_{3,4} \) as a minor. The sixth graph contains \( \mathfrak{Q}^+ \) as a minor.}
	\label{fig:F4+X}
\end{figure}

\begin{lemma} \label{lem:F4+H}
	Let \( e_1 \) and \( e_2 \) be two independent edges of \( F_4 \). Let \( H \) be the graph obtained from \( F_4 \) by subdividing each of \( e_1 \) and \( e_2 \) with a new vertex and then adding an edge between these two new vertices. Then \( H \) contains \( K_{3,4} \) as a minor.
\end{lemma}
\begin{proof}
Suppose one of \( e_1, e_2 \) is \( f^1 f^1_1 \), \( f^1 f^1_4 \), \( f^2 f^2_1 \), or \( f^2 f^2_4 \). Without loss of generality, assume that \( e_1 \) is \( f^1 f^1_1 \). If \( e_2 \) is \( f^1_2 f^1_3 \) or \( f^1_3 f^2_2 \), then \( H \) contains the second graph in Figure~\ref{fig:F4sub} as a minor. If \( e_2 \) is \( f^1_3 f^1_4 \), then \( H \) contains the first graph in Figure~\ref{fig:F4sub} as a minor. If \( e_2 \) is \( f^1_4 f^2_1 \), then \( H \) contains the second graph in Figure~\ref{fig:F4+X} as a minor. If \( e_2 \) is incident to \( f^2_3 \), then \( H \) contains the first graph in Figure~\ref{fig:F4+X} as a minor. If \( e_2 \) is incident to \( f^2 \), then \( H \) contains the sixth graph in Figure~\ref{fig:F4sub} as a minor. This exhausts all possibilities, and in each case, \( H \) has a \( K_{3,4} \) minor.

Suppose one of \( e_1, e_2 \) is \( f^1 f^1_2 \) or \( f^2 f^2_2 \). Without loss of generality, assume that \( e_1 \) is \( f^1 f^1_2 \). If \( e_2 \) is \( f^1_1 f^1_3 \) or \( f^1_4 f^1_3 \), then \( H \) contains the second graph in Figure~\ref{fig:F4sub} as a minor. If \( e_2 \) is incident to \( f^2_1 \) or \( f^2_4 \), then \( H \) contains the third graph in Figure~\ref{fig:F4sub} as a minor. If \( e_2 \) is incident to \( f^2_2 \), then \( H \) contains the fourth graph in Figure~\ref{fig:F4sub} as a minor. In all cases, \( H \) has a \( K_{3,4} \) minor.

	Therefore, we can assume that neither \( e_1 \) nor \( e_2 \) is incident to \( f^1 \) or \( f^2 \).
	
	Suppose one of \( e_1, e_2 \) is one of \( f^1_3 f^1_1 \), \( f^1_3 f^1_4 \), \( f^2_3 f^2_1 \), or \( f^2_3 f^2_4 \). Without loss of generality, we assume \( e_1 \) is \( f^1_3 f^1_1 \). If \( e_2 \) is \( f^1_4 f^2_1 \), then \( H \) contains the fifth graph in Figure~\ref{fig:F4+X} as a minor. If \( e_2 \) is incident to \( f^2_3 \), then \( H \) contains the third graph in Figure~\ref{fig:F4+X} as a minor. In any case, \( H \) has a \( K_{3,4} \) minor.
	
	Suppose one of \( e_1, e_2 \) is one of \( f^1_2 f^1_3 \) or \( f^2_2 f^2_3 \). Without loss of generality, we assume \( e_1 \) is \( f^1_2 f^1_3 \). If \( e_2 \) is incident to \( f^2_1 \) or \( f^2_4 \), then \( H \) contains the third graph in Figure~\ref{fig:F4sub} as a minor. If \( e_2 \) is \( f^2_2 f^2_3 \), then \( H \) contains the fourth graph in Figure~\ref{fig:F4sub} as a minor. In any case, \( H \) has a \( K_{3,4} \) minor.
	
By symmetry, it remains to consider the cases where \( \{e_1, e_2\} \) is \( \{f^1_1 f^2_4, f^1_2 f^2_3\} \), \( \{f^1_1 f^2_4, f^1_3 f^2_2\} \), \( \{f^1_1 f^2_4, f^1_4 f^2_1\} \), or \( \{f^1_2 f^2_3, f^1_3 f^2_2\} \). In these cases, \( H \) contains either the third graph in Figure~\ref{fig:F4sub}, the fourth graph in Figure~\ref{fig:F4+X}, or the fourth graph in Figure~\ref{fig:F4sub} as a minor, and thus has a \( K_{3,4} \) minor.
\end{proof}

\begin{lemma}\label{lem:F4+T}
	Let \( u_1, u_2, w_1, w_2 \in V(F_4) \) be distinct vertices of $F_4$ such that \( w_1 u_1, u_1 u_2, u_2 w_2 \in E(F_4) \). Let $H$ be obtained from $F_4$ by subdividing $u_1 u_2$ with a vertex $w_3$, adding a vertex $v$, and joining $v$ to $w_1$, $w_2$, and $w_3$. Then $H$ contains a minor of $K_{3,4}$.
\end{lemma}
\begin{proof}
	Denote $e := u_1 u_2$. By Lemma~\ref{lem:F4+H}, we can assume that $w_1 w_2 \neq E(F_4)$.
	
	By symmetry, there are six cases.
	
	If \( e = f^1 f^1_1 \), then we have \( \{w_1, w_2\} = \{f^1_2, f^2_4\} \) or \( \{w_1, w_2\} = \{f^1_4, f^2_4\} \), corresponding respectively to \( H \) containing a minor of the third graph in Figure~\ref{fig:F4sub} and the second graph in Figure~\ref{fig:F4+X}.  
	
	If \( e = f^1 f^1_2 \), then \( \{w_1, w_2\} = \{f^1_1, f^2_3\} \) or \( \{w_1, w_2\} = \{f^1_4, f^2_3\} \), and \( H \) contains the first graph in Figure~\ref{fig:F4+X} as a minor. 
	
	If \( e = f^1_1 f^1_3 \), then \( \{w_1, w_2\} \) is one of \( \{f^1, f^2_2\} \), \( \{f^2_4, f^1_2\} \), \( \{f^2_4, f^1_4\} \), or \( \{f^2_4, f^2_2\} \), and \( H \) contains the fifth or third graph in Figure~\ref{fig:F4sub} or the fifth graph in Figure~\ref{fig:F4+X} as a minor. 
	
	If \( e = f^1_2 f^1_3 \), then \( \{w_1, w_2\} \) is one of \( \{f^1, f^2_2\} \), \( \{f^2_3, f^1_1\} \), or \( \{f^2_3, f^1_4\} \), and \( H \) contains the fifth or second graph in Figure~\ref{fig:F4sub} as a minor. 
	
	If \( e = f^1_1 f^2_4 \), then \( \{w_1, w_2\} \) is one of \( \{f^1, f^2\} \), \( \{f^1, f^2_3\} \), \( \{f^1_3, f^2\} \), or \( \{f^1_3, f^2_3\} \), and \( H \) contains the sixth graph in Figure~\ref{fig:F4sub} or the first or third graph in Figure~\ref{fig:F4+X} as a minor. 
	
	Finally, if \( e = f^1_2 f^2_3 \), then \( \{w_1, w_2\} \) is one of \( \{f^1, f^2_1\} \), \( \{f^1, f^2_2\} \), \( \{f^1, f^2_4\} \), \( \{f^1_3, f^2_1\} \), or \( \{f^1_3, f^2_4\} \), and \( H \) contains the third or fifth graph in Figure~\ref{fig:F4sub} as a minor. 
	
	In any case, \( H \) contains a \( K_{3,4} \) minor. 
\end{proof}

\begin{proposition} \label{pro:F4span}
	Let $G$ be an internally $4$-connected graph without $K_{3,4}$ minors. Suppose $G$ has no $\mathfrak{Q}^+$ minor or $G$ is $4$-connected. Then every JT-subdivision $\eta(F_4)$ of $F_4$ in $G$ is a spanning subgraph of $G$.
\end{proposition}
\begin{proof}
	Suppose, to the contrary, that $\eta(F_4)$ does not span $G$. There exists a bridge $B$ of $\eta(F_4)$. 
	
	Since \( G \) and \( F_4 \) are internally 4-connected and Lemma~\ref{lem:F4sub} holds, we can apply Lemma~\ref{lem:JTstable} to conclude that \( B \) is a stable bridge.

	We first show that no two attachments of \( B \) belong to the same segment. Suppose this does not hold, and there exists a segment \( \eta(e) \) containing two attachments of \( B \). By Lemma~\ref{lem:F4+H}, if \( B \) has an attachment in some segment disjoint from \( \eta(e) \), then that attachment must be in \( \eta(V(F_4)) \). Therefore, by the stability of \( B \), it is not hard to show that \( G \) contains a minor of a graph obtained from \( F_4 \) by adding a vertex and joining it to four vertices of \( F_4 \). It follows from Lemma~\ref{lem:F4+X} that \( G \) contains a \( K_{3,4} \) minor, which is a contradiction. Hence, no segment can contain two attachments of $B$.

We first assume $G$ has no $\mathfrak{Q}^+$ minor. 
	
We claim that there exists an attachment of \( B \) outside \( \eta(V(F_4)) \). Suppose not. Then, by applying Lemma~\ref{lem:F4+Y} and using the fact that no two attachments of \( B \) lie in the same segment, we conclude that \( G \) contains a \( K_{3,4} \) minor, a contradiction.

Let \( v_1, v_2, v_3 \) be three attachments of \( B \) such that one of them is not in \( \eta(V(F_4)) \)

	We now show that it causes no loss of generality to assume the existence of distinct vertices \( u_1, u_2, w_1, w_2 \in V(F_4) \) such that \( w_1 u_1, u_1 u_2, u_2 w_2 \in E(F_4) \), \( v_3 \) is an internal vertex of \( \eta(u_1 u_2) \), and, for \( i \in [2] \), \( v_i \) is in the bark of $w_i$.
	
	We can assume that \( v_3 \) is in \( \eta(u_1 u_2) - \eta(u_1) - \eta(u_2) \) for some edge \( u_1 u_2 \in E(F_4) \). Note that \( v_1 \) and \( v_2 \) are not in the segment \( \eta(u_1 u_2) \).
	
	By appropriate edge contractions, we can move \( v_1 \) and \( v_2 \) to two distinct vertices in \( \eta(V(F_4)) \setminus \{\eta(u_1), \eta(u_2)\} \). More precisely, for \( i \in [2] \), we do nothing if \( v_i \) is already in \( \eta(V(F_4)) \setminus \{\eta(u_1), \eta(u_2)\} \); otherwise, letting \( \eta(e_i) \) be the segment containing \( v_i \), we can contract a subpath of \( \eta(e_i) \) that joins \( v_i \) to an end-vertex of \( \eta(e_i) \) other than \( \eta(u_1) \), \( \eta(u_2) \). Since \( F_4 \) has no cycle of length 3, this operation always allows us to move \( v_1 \) and \( v_2 \) to two distinct vertices in \( \eta(V(F_4)) \setminus \{\eta(u_1), \eta(u_2)\} \).
	
	Let \( w_1 \) and \( w_2 \) be vertices in \( V(F_4) \setminus \{u_1, u_2\} \) such that \( v_1 \) and \( v_2 \) are moved to \( \eta(w_1) \) and \( \eta(w_2) \), respectively. Whenever possible, we choose the moves such that \( w_1 \) and \( w_2 \) are non-adjacent vertices in \( V(F_4) \setminus \{u_1, u_2\} \).
	
	If \( w_1 \) and \( w_2 \) are non-adjacent in \( F_4 \), then it follows from Lemma~\ref{lem:F4+Y} that each of $u_1$ and $u_2$ is adjacent to one of $w_1$ or $w_2$. This yields the desired configuration as $v_i$ is in the bark of $w_i$.
	
	If \( w_1 \) and \( w_2 \) are adjacent in \( F_4 \), then, since \( v_1 \) and \( v_2 \) are not in the same segment, one of them—say \( v_1 \)—must be moved, and it has only one possible relocation. This happens only when \( v_1 \) is not in \( \eta(V(F_4)) \) and is in a segment that has either \( \eta(u_1) \) or \( \eta(u_2) \) as an end-vertex. In this case, we obtain the desired configuration by letting \( v_1 \) play the role of \( v_3 \).

	With this observation, we have that $G$ contains a minor of a graph that is obtained from $F_4$ by subdividing an edge $e = u_1 u_2 \in E(F_4)$ with a new vertex $w_3$, adding another new vertex $v$, and joining $v$ to $w_1, w_2, w_3$, where, for $i \in [2]$, $w_i$ is a neighbor of $u_i$ other than $u_{3-i}$ in $F_4$. By Lemma~\ref{lem:F4+T}, $G$ contains a minor of $K_{3,4}$, which is a contradiction.
	
Finally, it remains to address the case where \( G \) is 4-connected. In this case, \( B \) has at least four attachments. By arguments similar to those used in the proof of Proposition~\ref{pro:E20span}, one can easily show that \( G \) contains as a minor a graph obtained from \( F_4 \) by adding a vertex and joining it to four vertices. The resulting contradiction—that \( G \) has a \( K_{3,4} \) minor—then follows from Lemma~\ref{lem:F4+X}.
\end{proof}

\section{Minimal minors of 3-connected non-hamiltonian graphs}\label{sec:DMC}

In this section, we prove Ding and Marshall's conjecture.

Throughout, we assume that \( G \) is a minor-minimal counterexample to Conjecture~\ref{con:DM}.

Recall that, as shown in Section~\ref{sec:projective-planar}, \( G \) is not projective-planar.

Refining a classical result of Archdeacon~\cite{Archdeacon1981}, Robertson, Seymour, and Thomas proved that every 3-connected non-projective-planar graph contains a minor of one of the 23 graphs listed in Appendix~\ref{sec:A3}~\cite{Ding2014}. (Those 23 graphs are the 3-connected graphs from Archdeacon's list.) It is easy to see that all of these graphs, except \( D_{17} \), \( E_{20} \), and \( F_4 \), have a minor of \( K_{3,4} \) or \( \mathfrak{Q}^+ \). Therefore, as \( G \) is an internally 4-connected non-projective-planar graph without minors of $K_{3,4}$ or $\mathfrak{Q}^+$, it must contain a minor of \( D_{17} \), \( E_{20} \), or \( F_4 \).

Since $G$ is non-hamiltonian and $D_{17}$ is hamiltonian, if $G$ contains a minor of $D_{17}$, then $|V(G)| > |V(D_{17})|$. By Proposition~\ref{pro:D17toE20}, Lemma~\ref{lem:na4}, Proposition~\ref{pro:E20toF4DM}, and Lemma~\ref{lem:F4sub}, we conclude that $G$ contains a subdivision of $F_4$. Denote by $\eta(F_4)$ a JT-subdivision of $F_4$ in $G$. By Proposition~\ref{pro:F4span}, $\eta(F_4)$ is a spanning subgraph of $G$.

We now show some properties concerning internal vertices of segments of $\eta$.

\begin{claim} \label{cla:noint1}
	Let \( e \in E(F_4) \) be one of \( f^1_1 f^1_3 \), \( f^1_3 f^1_4 \), \( f^2_1 f^2_3 \), or \( f^2_3 f^2_4 \). Then the segment \( \eta(e) \) has no internal vertex.
\end{claim}
\begin{proof}
	It suffices to prove the claim for \( e = f^1_1 f^1_3 \).
		
	Suppose, to the contrary, that \( \eta(e) \) has an internal vertex. Let \( v \) be the neighbor of \( \eta(f^1_1) \) in \( \eta(e) \). By Lemma~\ref{lem:JTstable} and the fact that \( G \) is 3-connected, we conclude that \( v \) is adjacent to some vertex \( u \) outside \( \eta(e) \).

	The vertex \( u \) is not in the bark of \( f^1_4 \), \( f^1_2 \), \( f^2_2 \), \( f^2_3 \), \( f^2_1 \), or \( f^2 \); otherwise, \( G \) would contain a minor of the first, second, or third graph in Figure~\ref{fig:F4sub}, or the third, fifth, or sixth graph in Figure~\ref{fig:F4+X}.
		
	If \( u \) lies in \( \eta(f^1 f^1_1) - \eta(f^1_1) \), then by Lemma~\ref{lem:JTedge}, \( u \) is adjacent to \( \eta(f^1_1) \). This would imply that \( \{v, u, \eta(f^1_1)\} \) induces a cycle of length three, contradicting Lemma~\ref{lem:n3c}. Similarly, if \( u \) lies in \( \eta(f^1_1 f^2_4) - \eta(f^1_1) \), then by Lemma~\ref{lem:JTedge}, \( u \) is adjacent to \( \eta(f^1_1) \), again contradicting Lemma~\ref{lem:n3c}.

	Thus, \( u \) must be in \( \eta(e) \), which is impossible.
\end{proof}

\begin{claim} \label{cla:noint2}
	Let \( e \in E(F_4) \) be \( f^1_2 f^1_3 \) or \( f^2_2 f^2_3 \). Then the segment \( \eta(e) \) has no internal vertex.
\end{claim}
\begin{proof}
	It suffices to prove the claim for \( e = f^1_2 f^1_3 \).
	
	Suppose, for contradiction, that \( \eta(e) \) has an internal vertex. Let \( v \) be the neighbor of \( \eta(f^1_2) \) in \( \eta(e) \). By Lemma~\ref{lem:JTstable} and the fact that $G$ is 3-connected, we deduce that \( v \) is adjacent to some vertex \( u \) outside \( \eta(e) \).
	
	Note that \( u \) cannot lie in the bark of \( f^1_1 \), \( f^1_4 \), \( f^2 \), \( f^2_1 \), \( f^2_2 \), or \( f^2_4 \); otherwise, \( G \) would contain a minor of the second, third, fourth, or fifth graph in Figure~\ref{fig:F4sub}.
	We conclude that \( u \) is in \( \eta(f^1 f^1_2) - \eta(f^1_2) \) or \( \eta(f^1_2 f^2_3) - \eta(f^1_2) \). Then, by Lemma~\ref{lem:JTedge}, \( u \) is adjacent to \( \eta(f^1_2) \) within these segments, contradicting Lemma~\ref{lem:n3c}.
\end{proof}

For \( i \in [2] \) and \( j \in \{1, 2, 4\} \), let \( y^i_j \) and \( z^i_j \) denote the neighbor of \( \eta(f^i) \) and that of \( \eta(f^i_j) \) in \( \eta(f^i f^i_j) \), respectively.
				
\begin{claim} \label{cla:int}
	Let \( i \in [2] \) and let \( e \in E(F_4) \) be \( f^i f^i_j \) with \( j \in \{1, 4\} \). Then \( z^i_j \) is adjacent to \( y^i_2 \) or \( y^i_{5-j} \).
\end{claim}
					
\begin{proof}
	It suffices to prove the claim for \( e = f^1 f^1_1 \). Clearly, it holds when \( \eta(e) \) has no internal vertex, that is, \( z^1_1 = \eta(f^1) \). So we assume \( z^1_1 \neq \eta(f^1) \) is an internal vertex of \( \eta(e) \).
	
	By Lemma~\ref{lem:JTstable} and the fact that \( G \) is 3-connected, we have that \( z^1_1 \) is adjacent to some vertex \( u \) outside \( \eta(e) \).
	
	Note that \( u \) cannot lie in the bark of \( f^2 \), \( f^2_1 \), \( f^2_2 \), or \( f^2_3 \); otherwise, \( G \) would contain a minor of the third or sixth graph in Figure~\ref{fig:F4sub}, or the first or second graph in Figure~\ref{fig:F4+X}. Also, it follows from Claims~\ref{cla:noint1} and~\ref{cla:noint2} that neither \( \eta(f^1_2 f^1_3) \) nor \( \eta(f^1_3 f^1_4) \) contains an internal vertex.
	
	If \( u \) belongs to \( \eta(f^1_1 f^1_3) - \eta(f^1_1) \) or \( \eta(f^1_1 f^2_4) - \eta(f^1_1) \), then, by Lemma~\ref{lem:JTedge}, \( u \) is adjacent to \( \eta(f^1_1) \) within these segments, contradicting Lemma~\ref{lem:n3c}.
	
	We conclude that \( u \) is in \( \eta(f^1 f^1_2) - \eta(f^1) \) or \( \eta(f^1 f^1_4) - \eta(f^1) \). By Lemma~\ref{lem:JTedge}, \( u \) is either \( y^1_2 \) or \( y^1_4 \). This proves the claim.
\end{proof}
						
We are now ready to construct a Hamilton cycle of \( G \).
		
By Claim~\ref{cla:int}, $z^1_4$ and $z^2_1$ are adjacent to $u^1 \in \{y^1_1, y^1_2\}$ and $u^2 \in \{y^2_2, y^2_4\}$, respectively.
						
	By Claims~\ref{cla:noint1},~\ref{cla:noint2}, and the fact that $\eta(F_4)$ spans $G$, we conclude that the cycle obtained from the union of $\eta(f^1 f^1_1)$, $\eta(f^1_1 f^2_4)$, $\eta(f^2_4 f^2)$, $\eta(f^2 f^2_2)$, $\eta(f^2_2 f^1_3)$, $\eta(f^1_3 f^1_4)$, $\eta(f^1_4 f^2_1)$, $\eta(f^2_1 f^2_3)$, $\eta(f^2_3 f^1_2)$, and $\eta(f^1_2 f^1)$, by replacing the edges $\eta(f^1) u^1$ and $\eta(f^2) u^2$ with $\eta(f^1 f^1_4) - \eta(f^1_4)$, $z^1_4 u^1$, $\eta(f^2 f^2_1) - \eta(f^2_1)$ and $z^2_1 u^2$, is a Hamilton cycle of $G$, which is a contradiction.
	
	This proves the conjecture of Ding and Marshall.
						
\section{Every 4-connected non-hamiltonian graph contains a $K_{3,4}$ minor}\label{sec:DMQ}

In this section, we affirmatively answer Ding and Marshall's question.

Throughout, we assume that \( G \) is a counterexample to Question~\ref{que:DM}, meaning that \( G \) is a 4-connected non-hamiltonian graph without \( K_{3,4} \) minors.

By Thomas and Yu's theorem~\cite{Thomas1994}, \( G \) is non-projective-planar. Therefore, applying the result of Robertson, Seymour, and Thomas~\cite{Ding2014}, $G$ contains a minor of one of the 23 graphs listed in Appendix~\ref{sec:A3}. As $G$ does not have a $K_{3,4}$ minor, it must have a minor of $D_{17}$, $E_{20}$, $E_{22}$, or $F_4$.

Since $G$ is non-hamiltonian while $D_{17}$ is hamiltonian, having a $D_{17}$ minor in $G$ implies that $|V(G)| > |V(D_{17})|$. It follows from Proposition~\ref{pro:D17toE20} and Proposition~\ref{pro:E22} that $G$ contains a minor of $E_{20}$ or $F_4$.

\smallskip

Suppose $G$ contains no $F_4$ minor. By Lemma~\ref{lem:E20sub} and Proposition~\ref{pro:E20span}, $G$ contains a spanning JT-subdivision $\eta_1(E_{20})$ of $E_{20}$. We prove some properties concerning the subdivision $\eta_1(E_{20})$.

\begin{claim} \label{cla:noint3}
	Let \( e \in E(E_{20}) \) be one of \( e^0 e^1_1 \), \( e^0 e^1_2 \), or \( e^0 e^1_3 \). Then the segment \( \eta_1(e) \) has no internal vertex.
\end{claim}
\begin{proof}
	It suffices to prove the claim for \( e = e^0 e^1_1 \).
	
	Suppose, to the contrary, that \( \eta_1(e) \) has an internal vertex $v$. By Lemma~\ref{lem:JTstable} and the fact that \( G \) is 3-connected, we conclude that \( v \) is adjacent to some vertex \( u \) outside \( \eta_1(e) \).

	The vertex \( u \) is not in the bark of any vertex other than $e^0$ and $e^1_1$; otherwise, \( G \) would contain a minor of the first, fifth, or sixth graph in Figure~\ref{fig:E20sub} and hence a minor of $K_{3,4}$ or $F_4$.
	
	Thus, \( u \) must be in \( \eta_1(e) \), which is impossible.
\end{proof}

\begin{claim} \label{cla:noint4}
	Let \( e \in E(E_{20}) \) be one of \( e^1_1 e^1_2 \), $e^1_2 e^1_3$, or \( e^1_3 e^1_1 \). Then the segment \( \eta_1(e) \) has no internal vertex.
\end{claim}
\begin{proof}
	It suffices to prove the claim for \( e = e^1_1 e^1_2 \).
	
	Suppose, to the contrary, that \( \eta_1(e) \) has an internal vertex \( v \). By Lemma~\ref{lem:JTstable} and the fact that \( G \) is 4-connected, we conclude that \( v \) is adjacent to two vertices \( u_1 \) and \( u_2 \) outside \( \eta_1(e) \).
	
	Note that for each \( i \in [2] \), the vertex \( u_i \) is not in the bark of any vertex other than \( e^1_1 \), \( e^1_2 \), or \( e^4 \); otherwise, \( G \) would contain a minor of one of the first, second, third, fourth, or sixth graphs in Figure~\ref{fig:E20toF4e}, and hence a minor of \( K_{3,4} \) or \( F_4 \).
	
	Thus, both \( u_1 \) and \( u_2 \) must be the vertex \( \eta_1(e^4) \), which is absurd.
\end{proof}

For \( i \in [3] \), let \( u^1_i \) and \( u^2_i \) denote the neighbors of \( \eta_1(e^3_i) \) in \( \eta_1(e^1_i e^3_i) \) and \( \eta_1(e^2 e^3_i) \), respectively. Similarly, let \( v_i \) be the neighbor of \( \eta_1(e^2) \) in \( \eta_1(e^2 e^3_i) \), and let \( w_i \) be the neighbor of \( \eta_1(e^4) \) in \( \eta_1(e^4 e^3_i) \).

\begin{claim} \label{cla:intj}
	Let \( i \in [3] \) and \( e = e^2 e^3_i \in E(E_{20}) \). Suppose \( \eta_1(e) \) has an internal vertex \( v \). Then \( v \) is adjacent to both \( u^1_i \) and \( u^4_i \).
\end{claim}

\begin{proof}
	By Lemma~\ref{lem:JTstable} and the fact that $G$ is 4-connected, we conclude that \( v \) is adjacent to two vertices \( u_1 \) and \( u_2 \) outside \( \eta_1(e) \).
	
	For each \( j \in [2] \), the vertex \( u_j \) is not in the bark of any vertex other than \( e^1_i \), \( e^2 \), \( e^3_i \), or \( e^4 \); otherwise, \( G \) would contain a minor of the fourth or fifth graph in Figure~\ref{fig:E20sub} or the fifth graph in Figure~\ref{fig:E20+Y}, and hence a minor of \( K_{3,4} \) or \( F_4 \).
	
	This implies that $u_j$ is in \( \eta_1(e^1_i e^3_i) - \eta_1(e^3_i) \) or \( \eta_1(e^4 e^3_i) - \eta_1(e^3_i) \).
	
	By Lemma~\ref{lem:JTedge}, if \( u_j \) lies in \( \eta_1(e^1_i e^3_i) - \eta_1(e^3_i) \), then \( u_j = u^1_i \), and if \( u_j \) lies in \( \eta_1(e^4 e^3_i) - \eta_1(e^3_i) \), then \( u_j = u^4_i \). Thus, we conclude that \( u_1 \) and \( u_2 \) are precisely \( u^1_i \) and \( u^4_i \).
\end{proof}

\begin{claim} \label{cla:infj}
	Every internal vertex of \( \eta_1(e^3_3 e^4) \) is adjacent to \( w_1 \), \( w_2 \), \( u^1_3 \), or \( u^2_3 \).
\end{claim}

\begin{proof}
	Let \( v \) be an internal vertex of \( \eta_1(e^3_3 e^4) \). By Lemma~\ref{lem:JTstable} and the fact that \( G \) is 3-connected, we have that \( v \) is adjacent to some vertex \( u \) outside \( \eta_1(e^3_3 e^4) \).
	
	The vertex \( u \) is not in the bark of \( e^0 \), \( e^1_1 \), or \( e^1_2 \); otherwise, \( G \) would contain a minor of the fifth graph in Figure~\ref{fig:E20sub} or the fourth graph in Figure~\ref{fig:E20+Y}, and hence a minor of \( K_{3,4} \).
	
	By Lemma~\ref{lem:E20+H}, \( u \) is not an internal vertex of \( \eta_1(e^2 e^3_1) \) or \( \eta_1(e^2 e^3_2) \).
	
	Therefore, \( u \) is in \( \eta_1(e^4 e^3_1) - \eta_1(e^4) \), \( \eta_1(e^4 e^3_2) - \eta_1(e^4) \), \( \eta_1(e^2 e^3_3) - \eta_1(e^3_3) \), or \( \eta_1(e^1_3 e^3_3) - \eta_1(e^3_3) \). Applying Lemma~\ref{lem:JTedge}, we conclude that \( u \) is one of \( w_1 \), \( w_2 \), \( u^1_3 \), or \( u^2_3 \).
\end{proof}

Claim~\ref{cla:intj} implies that $u^1_1 v_1, u^1_2 v_2 \in E(G)$.

Claim~\ref{cla:infj} implies that the path \( \eta_1(e^4 e^3_3) \) contains two adjacent vertices \( x_1, x_2 \) such that \( x_1 \) is not in \( \eta_1(e^3_3 e^4)[x_2, \eta_1(e^4)] \), \( x_1 \) is adjacent to \( u^{i_1}_3 \), and \( x_2 \) is adjacent to \( w_{i_2} \) for some $i_1, i_2 \in [2]$.

It follows from Claims~\ref{cla:noint3} and~\ref{cla:noint4} that the cycle obtained from the union of $\eta_1(e^0 e^2)$, $\eta_1(e^2 e^3_3)$, $\eta_1(e^3_3 e^1_3)$, $\eta_1(e^1_3 e^1_1)$, $\eta_1(e^1_1 e^3_1)$, $\eta_1(e^3_1 e^4)$, $\eta_1(e^4 e^3_2)$, $\eta_1(e^3_2 e^1_2)$, and $\eta_1(e^1_2 e^0)$, by first removing the edges $\eta_1(e^3_1) u^1_1$, $\eta_1(e^3_2) u^1_2$, $\eta_1(e^3_3) u^{i_1}_3$, and $\eta_1(e^4) w_{i_2}$, and then adding $u^1_1 v_1$, $\eta_1(e^2 e^3_1) - \eta_1(e^2)$, $u^1_2 v_2$, $\eta_1(e^2 e^3_2) - \eta_1(e^2)$, $u^{i_1}_3 x_1$, \( \eta_1(e^3_3 e^4)[\eta_1(e^3_3), x_1] \), $w_{i_2} x_2$, and \( \eta_1(e^4 e^3_3)[\eta_1(e^4), x_2] \), is a Hamilton cycle of $G$, which is a contradiction.

\smallskip

Consequently, \( G \) must contain a minor of \( F_4 \).

We apply Lemma~\ref{lem:F4sub} and Proposition~\ref{pro:F4span} to deduce that $G$ contains a spanning JT-subdivision $\eta_2(F_4)$ of $F_4$.

For \( i \in [2] \) and \( j \in \{1, 2, 4\} \), let \( y^i_j \) denote the neighbor of \( \eta_2(f^i) \) in \( \eta_2(f^i f^i_j) \), \( z^i_j \) the neighbor of \( \eta_2(f^i_j) \) in \( \eta_2(f^i f^i_j) \), \( r^i_j \) the neighbor of \( \eta_2(f^i_3) \) in \( \eta_2(f^i_3 f^i_j) \), and \( s^i_j \) the neighbor of \( \eta_2(f^i_j) \) in \( \eta_2(f^i_j f^i_3) \).
For \( i \in [2] \) and \( j \in [4] \), let \( t^i_j \) the neighbor of \( \eta_2(f^i_j) \) in \( \eta_2(f^i_j f^{3-i}_{5-j}) \).

\begin{claim} \label{cla:noint5}
	Any internal vertex of \( \eta_2(f^i_j f^i_3) \) with \( i \in [2] \) and \( j \in \{1, 4\} \) is adjacent to \( z^i_j \) or \( t^i_j \).
\end{claim}

\begin{proof}
	Let \( v \) be an internal vertex of \( \eta_2(f^i_j f^i_3) \). By Lemma~\ref{lem:JTstable} and the fact that \( G \) is 4-connected, we conclude that \( v \) is adjacent to at least two vertices outside \( \eta_2(f^i_j f^i_3) \).
	
	Any neighbor \( u \) of \( v \) cannot belong to the bark of \( f^i_{5-j} \), \( f^i_2 \), \( f^{3-i}_2 \), \( f^{3-i}_3 \), or \( f^{3-i}_j \); otherwise, \( G \) would contain a minor of the first, second, or third graph in Figure~\ref{fig:F4sub}, or the third or fifth graph in Figure~\ref{fig:F4+X}.
	
	Consequently, \( u \) must lie in \( \eta_2(f^i f^i_j) - \eta_2(f^i_j) \), \( \eta_2(f^i_j f^{3-i}_{5-j}) - \eta_2(f^i_j) \), or \( \eta_2(f^{3-i}_{5-j} f^{3-i}) - \eta_2(f^{3-i}_{5-j}) \). By Lemma~\ref{lem:JTedge} and Lemma~\ref{lem:F4+H}, it follows that \( u \) must be \( z^i_j \), \( t^i_j \), or \( \eta_2(f^{3-i}) \).
	
	Thus, \( v \) is adjacent to one of \( z^i_j \) or \( t^i_j \).
\end{proof}

\begin{claim} \label{cla:noint6}
		Any internal vertex of \( \eta_2(f^i_2 f^i_3) \) with \( i \in [2] \) is adjacent to \( z^i_2 \) and \( t^i_2 \).
\end{claim}
\begin{proof}
	Let \( v \) be an internal vertex of \( \eta_2(f^i_2 f^i_3) \). By Lemma~\ref{lem:JTstable} and the fact that \( G \) is 4-connected, we conclude that \( v \) is adjacent to at least two vertices outside \( \eta_2(f^i_2 f^i_3) \).
	
	Note that any neighbor \( u \) of \( v \) cannot lie in the bark of \( f^i_1 \), \( f^i_4 \), \( f^{3-i} \), \( f^{3-i}_1 \), \( f^{3-i}_2 \), or \( f^{3-i}_4 \); otherwise, \( G \) would contain a minor of the second, third, fourth, or fifth graph in Figure~\ref{fig:F4sub}.
	
	We conclude that \( u \) must be in \( \eta(f^i f^i_2) - \eta(f^i_2) \) or \( \eta(f^i_2 f^{3-i}_3) - \eta(f^i_2) \). Then, by Lemma~\ref{lem:JTedge}, \( u \) is adjacent to \( \eta(f^i_2) \) within these segments. Hence, we have that \( v \) is adjacent to both \( z^i_2 \) and \( t^i_2 \).
\end{proof}

\begin{claim} \label{cla:intjj}
Any internal vertex of \( \eta_2(f^1 f^1_4) \) is adjacent to at least two of \( y^1_1 \), \( y^1_2 \), \( s^1_4 \), or \( t^1_4 \), and any internal vertex of \( \eta_2(f^2 f^2_1) \) is adjacent to at least two of \( y^2_2 \), \( y^2_4 \), \( s^2_1 \), or \( t^2_1 \).
\end{claim}

\begin{proof}
	It suffices to prove the first statement. 
	
	Let \( v \) be an internal vertex of \( \eta_2(f^1 f^1_4) \). By Lemma~\ref{lem:JTstable} and the fact that \( G \) is 4-connected, we conclude that \( v \) is adjacent to at least two vertices outside \( \eta_2(f^1 f^1_4) \).
	
	Let $u$ be a neighbor of $v$ outside \( \eta_2(f^1 f^1_4) \). Note that \( u \) cannot belong to the bark of \( f^2 \), \( f^2_2 \), \( f^2_3 \), or \( f^2_4 \); otherwise, \( G \) would contain a minor of the third or sixth graph in Figure~\ref{fig:F4sub} or the first or second graph in Figure~\ref{fig:F4+X}. Additionally, \( u \) cannot be an internal vertex of \( \eta_2(f^1_1 f^1_3) \) or \( \eta_2(f^1_2 f^1_3) \), by Lemma~\ref{lem:F4+H}.
	
	Consequently, \( u \) must lie in \( \eta_2(f^1 f^1_1) - \eta_2(f^1) \), \( \eta_2(f^1 f^1_2) - \eta_2(f^1) \), \( \eta_2(f^1_4 f^1_3) - \eta_2(f^1_4) \), or \( \eta_2(f^1_4 f^2_1) - \eta_2(f^1_4) \). Applying Lemma~\ref{lem:JTedge}, the claim follows.
\end{proof}

We are ready to construct a Hamilton cycle in \( G \).

Claim~\ref{cla:noint5} asserts that \( r^1_1 \) is adjacent to \( a^1_1 \in \{z^1_1, t^1_1\} \), and \( r^2_4 \) is adjacent to \( a^2_4 \in \{z^2_4, t^2_4\} \).

Claim~\ref{cla:noint6} asserts that \( r^1_2 t^1_2, r^2_2 t^2_2 \in E(G) \).

Claim~\ref{cla:intjj} implies that the path \( \eta_2(f^1 f^1_4) \) contains two adjacent vertices \( x^1, x^1_4 \) such that \( x^1 \) does not lie in \( \eta_2(f^1 f^1_4)[x^1_4, \eta_2(f^1_4)] \), \( x^1 \) is adjacent to \( b^1 \in \{y^1_1, y^1_2\} \), and \( x^1_4 \) is adjacent to \( b^1_4 \in \{s^1_4, t^1_4\} \). Similarly, the path \( \eta_2(f^2 f^2_1) \) contains two adjacent vertices \( x^2, x^2_1 \) such that \( x^2 \) does not lie in \( \eta_2(f^2 f^2_1)[x^2_1, \eta_2(f^2_1)] \), \( x^2 \) is adjacent to \( b^2 \in \{y^2_2, y^2_4\} \), and \( x^2_1 \) is adjacent to \( b^2_1 \in \{s^2_1, t^2_1\} \).

In what follows, we consider replacements of edges by paths, ignoring any replacement in which the replacing path consists solely of the edge itself; for instance, the replacement of the edge \( \eta_2(f^1_1) a^1_1 \) with the path consisting of \( \eta_2(f^1_1 f^1_3) - \eta_2(f^1_3) \) and \( r^1_1 a^1_1 \) is performed only when \( r^1_1 \neq \eta_2(f^1_1) \).

We claim that the subgraph of \( G \) obtained from the union of \( \eta_2(f^1 f^1_1) \), \( \eta_2(f^1_1 f^2_4) \), \( \eta_2(f^2_4 f^2) \), \( \eta_2(f^2 f^2_2) \), \( \eta_2(f^2_2 f^1_3) \), \( \eta_2(f^1_3 f^1_4) \), \( \eta_2(f^1_4 f^2_1) \), \( \eta_2(f^2_1 f^2_3) \), \( \eta_2(f^2_3 f^1_2) \), and \( \eta_2(f^1_2 f^1) \), becomes a Hamilton cycle of \( G \) after performing the following replacements one after another: replace the edge \( \eta_2(f^1_1) a^1_1 \) with the path consisting of \( \eta_2(f^1_1 f^1_3) - \eta_2(f^1_3) \) and \( r^1_1 a^1_1 \); replace the edge \( \eta_2(f^2_4) a^2_4 \) with the path consisting of \( \eta_2(f^2_4 f^2_3) - \eta_2(f^2_3) \) and \( r^2_4 a^2_4 \); replace the edge \( \eta_2(f^1_2) t^1_2 \) with the path consisting of \( \eta_2(f^1_2 f^1_3) - \eta_2(f^1_3) \) and \( r^1_2 t^1_2 \); replace the edge \( \eta_2(f^2_2) t^2_2 \) with the path consisting of \( \eta_2(f^2_2 f^2_3) - \eta_2(f^2_3) \) and \( r^2_2 t^2_2 \); replace the edge \( \eta_2(f^1) b^1 \) with the path consisting of \( \eta_2(f^1 f^1_4)[\eta_2(f^1), x^1] \) and \( x^1 b^1 \); replace the edge \( \eta_2(f^1_4) b^1_4 \) with the path consisting of \( \eta_2(f^1 f^1_4)[x^1_4, \eta_2(f^1_4)] \) and \( x^1_4 b^1_4 \); replace the edge \( \eta_2(f^2) b^2 \) with the path consisting of \( \eta_2(f^2 f^2_1)[\eta_2(f^2), x^2] \) and \( x^2 b^2 \); and finally, replace the edge \( \eta_2(f^2_1) b^2_1 \) with the path consisting of \( \eta_2(f^2 f^2_1)[x^2_1, \eta_2(f^2_1)] \) and \( x^2_1 b^2_1 \).

To prove this claim, it suffices to verify that the edge being replaced is still present in the subgraph at the moment when the replacement is performed. We illustrate this by considering two representative cases; the remaining ones are similar and are omitted.

Suppose the edge \( \eta_2(f^2_4) a^2_4 \) is missing when we attempt to replace it. This could only happen if it was already removed during the earlier replacement of \( \eta_2(f^1_1) a^1_1 \). In that case, we must have \( a^1_1 = \eta_2(f^2_4) \neq r^2_4 \) and \( a^2_4 = \eta_2(f^1_1) \neq r^1_1 \), which implies that \( G \) contains the graph shown in Figure~\ref{fig:F4L} and hence \( K_{3,4} \) as a minor---a contradiction.

\begin{figure}[!ht]
	\centering{%
		\includegraphics[scale=1]{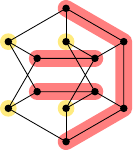}
	}
	\caption{A graph that contains \( K_{3,4} \) as a minor.}
	\label{fig:F4L}
\end{figure}

Similarly, if the edge \( \eta_2(f^1) b^1 \) is missing at the time we attempt to replace it, then it must be that \( a^1_1 = \eta_2(f^1) \neq x^1 \) and \( b^1 = \eta_2(f^1_1) \neq r^1_1 \), implying that \( G \) contains the second graph in Figure~\ref{fig:F4sub}, again leading to a contradiction.

We have thus shown that \( G \) is hamiltonian, leading to a contradiction. This completes the proof that Ding and Marshall's question has an affirmative answer---every 4-connected non-hamiltonian graph contains a \( K_{3,4} \) minor.

\section*{Acknowledgements}

The author wishes to thank Guoli Ding, Jorik Jooken, and Kenta Ozeki for helpful discussions.

\bibliographystyle{abbrv}
\bibliography{paper}

\begin{thebibliography}{10}

\bibitem{Archdeacon1981}
D.~Archdeacon.
\newblock A {Kuratowski} theorem for the projective plane.
\newblock {\em J. Graph Theory}, 5(3):243--246, 1981.

\bibitem{Chen2002}
G.~Chen and X.~Yu.
\newblock Long cycles in 3-connected graphs.
\newblock {\em J. Combin. Theory Ser. B}, 86(1):80--99, 2002.

\bibitem{Chen2012}
G.~Chen, X.~Yu, and W.~Zang.
\newblock The circumference of a graph with no ${K}_{3,t}$-minor, {II}.
\newblock {\em J. Combin. Theory Ser. B}, 102(6):1211--1240, 2012.

\bibitem{Ding2014}
G.~Ding and P.~Iverson.
\newblock Internally 4-connected projective-planar graphs.
\newblock {\em J. Comb. Theory Ser. B}, 108:123--138, 2014.

\bibitem{Ding2018}
G.~Ding and E.~Marshall.
\newblock Minimal $k$-{C}onnected {N}on-{H}amiltonian {G}raphs.
\newblock {\em Graphs Combin.}, 34(2):289--312, 2018.

\bibitem{Gruenbaum1970}
B.~Grünbaum.
\newblock Polytopes, graphs, and complexes.
\newblock {\em Bull. Amer. Math. Soc.}, 76:1131--1201, 1970.

\bibitem{Hall1943}
D.~W. Hall.
\newblock A note on primitive skew curves.
\newblock {\em Bull. Amer. Math. Soc.}, 49:935--937, 1943.

\bibitem{Johnson2002}
T.~Johnson and R.~Thomas.
\newblock Generating internally four-connected graphs.
\newblock {\em J. Combin. Theory Ser. B}, 85(1):21--58, 2002.

\bibitem{Lo2024}
O.-H.~S. Lo and K.~Ozeki.
\newblock Minors of non-hamiltonian polyhedra and the {H}erschel family.
\newblock Manuscript.

\bibitem{Maharry2016}
J.~Maharry and N.~Robertson.
\newblock The structure of graphs not topologically containing the {W}agner
  graph.
\newblock {\em J. Combin. Theory Ser. B}, 121:398--420, 2016.

\bibitem{Maharry2017}
J.~Maharry, N.~Robertson, V.~Sivaraman, and D.~Slilaty.
\newblock Flexibility of projective-planar embeddings.
\newblock {\em J. Combin. Theory Ser. B}, 122:241--300, 2017.

\bibitem{Maharry2012}
J.~Maharry and D.~Slilaty.
\newblock Projective-planar graphs with no {$K_{3,4}$}-minor.
\newblock {\em J. Graph Theory}, 70(2):121--134, 2012.

\bibitem{Maharry2017a}
J.~Maharry and D.~Slilaty.
\newblock Projective-planar graphs with no {$K_{3,4}$}-minor. {II}.
\newblock {\em J. Graph Theory}, 86(1):92--103, 2017.

\bibitem{Nash-Williams1973}
C.~{\relax St}. J.~A. Nash-Williams.
\newblock Unexplored and semi-explored territories in graph theory.
\newblock In {\em New directions in the theory of graphs}, pages 149--186.
  Academic Press, New York, 1973.

\bibitem{Ozeki2021}
K.~Ozeki.
\newblock Hamiltonicity of graphs on surfaces in terms of toughness and
  scattering number--a survey.
\newblock In J.~Akiyama, R.~M. Marcelo, M.-J.~P. Ruiz, and Y.~Uno, editors,
  {\em Discrete and computational geometry, graphs, and games (JCDCGGG 2018)},
  volume 13034 of {\em Lecture Notes in Computer Science}, pages 74--95.
  Springer Cham, 2021.

\bibitem{Seymour1980}
P.~D. Seymour.
\newblock Decomposition of regular matroids.
\newblock {\em J. Combin. Theory Ser. B}, 28:305--359, 1980.

\bibitem{Thomas1994}
R.~Thomas and X.~Yu.
\newblock 4-connected projective-planar graphs are hamiltonian.
\newblock {\em J. Combin. Theory Ser. B}, 62:114–132, 1994.

\bibitem{Tutte1956}
W.~T. Tutte.
\newblock A theorem on planar graphs.
\newblock {\em Trans. Amer. Math. Soc.}, 82:99–116, 1956.

\bibitem{Wagner1936}
K.~Wagner.
\newblock Bemerkungen zum {V}ierfarbenproblem.
\newblock {\em Jber. Deutsch. Math.-Verein.}, 46:26--32, 1936.

\bibitem{Wagner1937}
K.~Wagner.
\newblock {\" U}ber eine {E}rweiterung eines {S}atzes von {K}uratowski.
\newblock {\em Deutsche Math.}, 1937.

\end{thebibliography}

\appendix
\section{The minor-minimal 3-connected non-projective-planar graphs} \label{sec:A3}

\begin{table}[H]
	\centering
	\begin{tabular}{cccccc}
		\begin{minipage}{0.14\textwidth}
			\centering
			\includegraphics[scale=0.9]{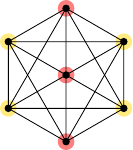} \\
			\caption*{$A_2$}
		\end{minipage} &
		\begin{minipage}{0.14\textwidth}
			\centering
			\includegraphics[scale=0.9]{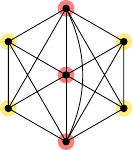} \\
			\caption*{$B_1$}
		\end{minipage} &
		\begin{minipage}{0.14\textwidth}
			\centering
			\includegraphics[scale=0.9]{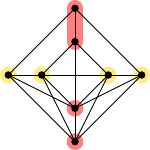} \\
			\caption*{$B_7$}
		\end{minipage} &
		\begin{minipage}{0.14\textwidth}
			\centering
			\includegraphics[scale=0.9]{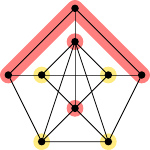} \\
			\caption*{$C_3$}
		\end{minipage} &
		\begin{minipage}{0.14\textwidth}
			\centering
			\includegraphics[scale=0.9]{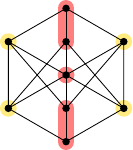} \\
			\caption*{$C_4$}
		\end{minipage} &
		\begin{minipage}{0.14\textwidth}
			\centering
			\includegraphics[scale=0.9]{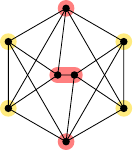} \\
			\caption*{$C_7$}
		\end{minipage}
	\end{tabular}
\end{table}

\begin{table}[H]
	\centering
	\begin{tabular}{cccccc}
		\begin{minipage}{0.14\textwidth}
			\centering
			\includegraphics[scale=0.9]{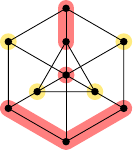} \\
			\caption*{$D_2$}
		\end{minipage} &
		\begin{minipage}{0.14\textwidth}
			\centering
			\includegraphics[scale=0.9]{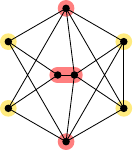} \\
			\caption*{$D_3$}
		\end{minipage} &
		\begin{minipage}{0.14\textwidth}
			\centering
			\includegraphics[scale=0.9]{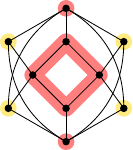} \\
			\caption*{$D_9$}
		\end{minipage} &
		\begin{minipage}{0.14\textwidth}
			\centering
			\includegraphics[scale=0.9]{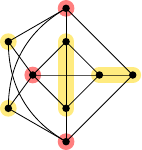} \\
			\caption*{$D_{12}$}
		\end{minipage} &
		\begin{minipage}{0.14\textwidth}
			\centering
			\includegraphics[scale=0.9]{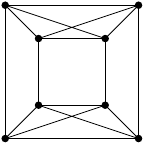} \\
			\caption*{$D_{17}$}
		\end{minipage} &
		\begin{minipage}{0.14\textwidth}
			\centering
			\includegraphics[scale=0.9]{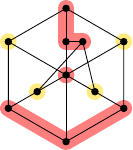} \\
			\caption*{$E_2$}
		\end{minipage}
	\end{tabular}
\end{table}

\begin{table}[H]
	\centering
	\begin{tabular}{cccccc}
		\begin{minipage}{0.14\textwidth}
			\centering
			\includegraphics[scale=0.9]{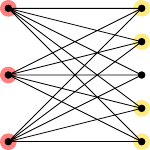} \\
			\caption*{$E_3$}
		\end{minipage} &
		\begin{minipage}{0.14\textwidth}
			\centering
			\includegraphics[scale=0.9]{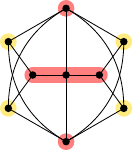} \\
			\caption*{$E_5$}
		\end{minipage} &
		\begin{minipage}{0.14\textwidth}
			\centering
			\includegraphics[scale=0.9]{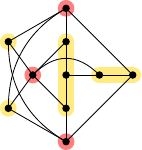} \\
			\caption*{$E_{11}$}
		\end{minipage} &
		\begin{minipage}{0.14\textwidth}
			\centering
			\includegraphics[scale=0.9]{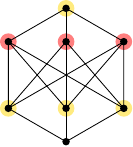} \\
			\caption*{$E_{18}$}
		\end{minipage} &
		\begin{minipage}{0.14\textwidth}
			\centering
			\includegraphics[scale=0.9]{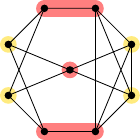} \\
			\caption*{$E_{19}$}
		\end{minipage} &
		\begin{minipage}{0.14\textwidth}
			\centering
			\includegraphics[scale=0.9]{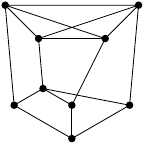} \\
			\caption*{$E_{20}$}
		\end{minipage}
	\end{tabular}
\end{table}

\begin{table}[H]
	\centering
	\begin{tabular}{cccccc}
		\begin{minipage}{0.14\textwidth}
			\centering
			\includegraphics[scale=0.9]{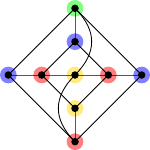} \\
			\caption*{$E_{22}$}
		\end{minipage} &
		\begin{minipage}{0.14\textwidth}
			\centering
			\includegraphics[scale=0.9]{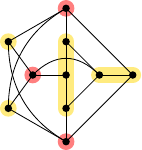} \\
			\caption*{$E_{27}$}
		\end{minipage} &
		\begin{minipage}{0.14\textwidth}
			\centering
			\includegraphics[scale=0.9]{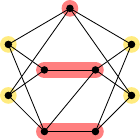} \\
			\caption*{$F_1$}
		\end{minipage} &
		\begin{minipage}{0.14\textwidth}
			\centering
			\includegraphics[scale=0.9]{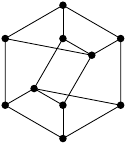} \\
			\caption*{$F_4$}
		\end{minipage} &
		\begin{minipage}{0.14\textwidth}
			\centering
			\includegraphics[scale=0.9]{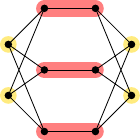} \\
			\caption*{$G_1$}
		\end{minipage} 
	\end{tabular}
\end{table}

\end{document}